\documentclass[12pt]{amsart}
\usepackage{amsthm}
\usepackage{amsmath,amsxtra,amssymb,amsfonts,latexsym, mathrsfs, dsfont,bm,mathtools}
\usepackage[initials]{amsrefs}
\usepackage[usenames]{color}
\usepackage[all]{xy}
\usepackage[hyperindex,pagebackref]{hyperref}

\makeatletter
\newcommand{\opnorm}{\@ifstar\@opnorms\@opnorm}
\newcommand{\@opnorms}[1]{%
  \left|\mkern-1.5mu\left|\mkern-1.5mu\left|
   #1
  \right|\mkern-1.5mu\right|\mkern-1.5mu\right|
}
\newcommand{\@opnorm}[2][]{%
  \mathopen{#1|\mkern-1.5mu#1|\mkern-1.5mu#1|}
  #2
  \mathclose{#1|\mkern-1.5mu#1|\mkern-1.5mu#1|}
}
\makeatother

\newtheorem{lemma}{Lemma}[section]
\newtheorem{prop}[lemma]{Proposition}
\newtheorem{theorem}[lemma]{Theorem}
\newtheorem{cor}[lemma]{Corollary}

\numberwithin{equation}{section}

\theoremstyle{remark}
\newtheorem{rem}[lemma]{Remark}

\theoremstyle{definition}

\newtheorem{exam}[lemma]{Example}

\newcommand{\lge}{\langle}
\newcommand{\rge}{\rangle}

\newcommand{\ax}{\mathcal{A}}
\newcommand{\bx}{\mathcal{B}}

\newcommand{\fx}{\mathcal{F}}
\newcommand{\gx}{\mathcal{G}}
\newcommand{\hx}{\mathcal{H}}
\newcommand{\kx}{\mathcal{K}}
\newcommand{\lx}{\mathcal{L}}
\newcommand{\mx}{\mathcal{M}}
\newcommand{\nx}{\mathcal{N}}
\newcommand{\rx}{\mathcal{R}}

\newcommand{\nz}{\mathbb{N}}
\newcommand{\rz}{\mathbb{R}}
\newcommand{\cz}{\mathbb{C}}
\newcommand{\ez}{\mathbb{E}}

\newcommand{\pz}{\mathbb{P}}

\newcommand{\Ga}{\Gamma}
\newcommand{\Om}{\Omega}

\newcommand{\al}{\alpha}
\newcommand{\de}{\delta}
\newcommand{\si}{\sigma}
\newcommand{\ga}{\gamma}
\newcommand{\la}{\lambda}

\newcommand{\ta}{\theta}

\newcommand{\eps}{\varepsilon}
\newcommand{\8}{\infty}

\mathchardef\dash="2D

\newcommand{\Aut}{\operatorname{Aut}}

\newcommand{\Ent}{\operatorname{Ent}}
\newcommand{\Lip}{\operatorname{Lip}}
\newcommand{\Fix}{\operatorname{Fix}}
\allowdisplaybreaks[1]
\title[Poincar\'e and transportation cost inequalities]{Poincar\'e type inequalities for group measure spaces and related transportation cost inequalities}
\author{Qiang Zeng}
\date{\today}
\address{Department of Mathematics, University of Illinois, Urbana, IL 61801}
\email{zeng8@illinois.edu}
\begin{document}

\begin{abstract}
Let $G$ be a countable discrete group with an orthogonal representation $\al$ on a real Hilbert space $H$. We prove $L_p$ Poincar\'e inequalities for the group measure space $L_\8(\Om_H,\ga)\rtimes G$, where both the group action and the Gaussian measure space $(\Om_H, \ga)$ are associated with the representation $\al$. The idea of proof comes from Pisier's method on the boundedness of Riesz transform and Lust-Piquard's work on spin systems. Then we deduce a transportation type inequality from the $L_p$ Poincar\'e inequalities in the general noncommutative setting. This inequality is sharp up to a constant (in the Gaussian setting). Several applications are given, including Wiener/Rademacher chaos estimation and new examples of Rieffel's compact quantum metric spaces.
\end{abstract}
\maketitle
\section{Introduction}
Gross' logarithmic Sobolev inequality (LSI) \cite{Gro} is a powerful tool and  has many applications in different areas of mathematics. Based on LSI, Bobkov and G\"otze proved the exponential integrability and the transportation cost inequality for 1-Wasserstein distance in \cite{BG}. It was proved later that LSI implies a stronger transportation cost inequality for 2-Wasserstein distance by Otto--Villani \cite{OV} and Bobkov--Gentil--Ledoux \cite{BGL}. It is well known now that the transportation cost inequalities imply concentration inequalities due to the work of Marton \cite{Mar}; see \cite{BG} and the references therein for more details in this direction. On the other hand, it is known that the $L_p$ Poincar\'e type inequalities also imply concentration phenomena. In \cite{ELP}, Efraim and Lust-Piquard proved the following $L_p$ Poincar\'e inequalities with constant $C\sqrt{p}$ for $2\le p<\8$ for the Walsh system
\[
\|f-\ez f\|_p\le C\sqrt{p} \| |\nabla f|\|_p
\]
where $\nabla f$ is the discrete gradient. Similar results hold for CAR algebras. Here and in what follows $C,C',C_1,c,c_1,etc.$ are absolute constants which may vary from line to line.  In the noncommutative setting, the $L_p$ Poincar\'e type inequalities with constant $Cp$ were proved for certain semigroups acting on finite von Neumann algebras under Bakry--Emery's $\Ga_2$-criterion in \cite{JZ}. The constant $C\sqrt{p}$ was obtained only with $L_\8$ norm on the right-hand side. The aim of this paper is to give consequences when the constant is $C\sqrt{p}$ and to provide more examples of this situation. We first prove the $L_p$ Poincar\'e inequalities for group (Gaussian) measure spaces, and then show that $L_p$ Poincar\'e inequalities with constant $C\sqrt{p}$ imply a transportation type inequality in general noncommutative setting. The latter seems to be new even in the commutative case. Moreover, it was observed in \cite{JZ} that LSI may fail in the non-diffusion setting but the Poincar\'e type inequalities still hold. This suggests that Poincar\'e type inequalities may be a simpler but more universal approach to transportation and concentration inequalities even though they usually provide less good constants compared with LSI. 

Unless specified otherwise, we consider a noncommutative probability space $(\nx, \tau)$ where $\nx$ is  a finite von Neumann algebra and $\tau$ a normal faithful tracial state. Then the noncommutative $L_p$ space $L_p(\nx,\tau)$ is the completion of $\nx$ with respect to $\|f\|_p=\tau[(f^*f)^{p/2}]^{1/p}$ for $0<p<\8$ and $\|f\|_\8=\|f\|$. Here and in the following $\|\cdot\|$ denotes the operator norm. It is well known that $L_p(\nx,\tau)$ is a Banach space for $1\le p \le \8$. For example, for a classical probability space $(\Om, \fx, \pz)$, we may take $\nx=L_\8(\Om, \pz)$ and $\tau(f)=\ez(f)=\int f d\pz$ for $f\in\nx$. Then $L_p(\nx,\tau)=L_p(\Om,\pz)$. Let $(T_t)_{t\ge0}$ be a \emph{standard} semigroup acting on $\nx$ with generator $A$, i.e., $T_t=e^{-tA}$. Following \cites{JMe, JZ}, a standard semigroup $(T_t)$ is a pointwise $\si$-weak (weak$*$) continuous semigroup such that every $T_t$ is  normal unital completely positive and self-adjoint on $L_2(\nx,\tau)$. The standard semigroup is a noncommutative analogue of a symmetric Markov semigroup in classical probability theory. Then the gradient form associated to $A$ (Meyer's ``carr\'e du champs'') is defined as
\[
\Ga_A(f_1,f_2)=\frac12[A(f_1^*)f_2+f_1^*A(f_2)-A(f_1^*f_2)]
\]
for $f_1,f_2$ in the domain of $A$. Let ${\rm Fix}=\{x\in \nx:T_t x=x\}$ be the fixed point algebra of $T_t$. It was shown in \cite{JX07} that ${\rm Fix}$ is a von Neumann subalgebra of $\nx$, thus there exists a unique conditional expectation $E_{\Fix}: \nx\to {\rm Fix}$. In this paper, we are interested in the following Poincar\'e type inequalities: for $2\le p<\8$
\begin{equation}\label{poin}
      \|x-E_{\rm Fix} x\|_p\le C\sqrt{p} \max\{ \|\Ga_A(x,x)^{1/2}\|_p,\|\Ga_A(x^*,x^*)^{1/2}\|_p\},
  \end{equation}
for $x\in\nx$.

In the first part of this paper, we consider a countable discrete group $G$ with an orthogonal representation $\al$ on a real Hilbert space $H$. From here we can construct a group measure space $L_\8(\rz^d,\ga_d)\rtimes_{\hat\al} G$, on which there is a canonical trace. Here $d$ is the dimension of $H$ and $\ga_d$ is the canonical product Gaussian measure. Put $\ell_2(G)={\rm span}\{\de_g: g\in G\}$ where $\de_g$ is the unit vector with $\de_g(g)=1$ and $\de_g(h)=0$ for $h\neq g$. Note that $L_\8(\rz^d,\ga_d)\rtimes_{\hat\al} G$ is a von Neumann subalgebra of $L_\8(\rz^d,\ga_d)\overline\otimes B(\ell_2(G))$. Here and in what follows $B(H)$ denotes the algebra of all bounded linear operators on a Hilbert space $H$. The construction of $L_\8(\rz^d,\ga_d)\rtimes_{\hat\al} G$ and other preliminary materials will be recalled in Section 2.

Consider the Ornstein--Uhlenbeck semigroup $(P_t)$ acting on $L_\8(\rz^d,\ga_d)$. Then $P_t\otimes id_{\ell_2(G)}$ acts on $L_\8(\rz^d,\ga_d)\overline\otimes B(\ell_2(G))$. By the construction of $L_\8(\rz^d,\ga_d) \rtimes_{\hat\al} G$, $P_t\otimes id_{\ell_2(G)}$ restricted to $L_\8(\rz^d,\ga_d) \rtimes_{\hat\al} G$ is a well-defined semigroup, and we denote the restriction by $T_t$. Then it is easy to check that $(T_t)$ is a standard semigroup on $L_\8(\rz^d,\ga_d) \rtimes_{\hat\al} G$ and $E_{\rm Fix} = E_{\lx(G)}$, where $\lx(G)$ is the group von Neumann algebra of $G$ constructed as follows. Let $\la: G\to B(\ell_2(G))$ be the left regular representation. Then the group von Neumann algebra $\lx(G)$ is the closure of linear span of $\la(G)$ in the weak operator topology. It is well known that $\lx(G)$ is a subalgebra of $L_\8(\rz^d,\ga_d) \rtimes_{\hat\al} G$ and admits a canonical normal faithful tracial state given by $\tau(f)=\lge \de_e,f\de_e\rge $ for $f\in\lx(G)$, where $e$ is the identity element of $G$. To be more specific, a generic element of $\lx(G)$ can be written as a Fourier series $f=\sum_g\hat f(g)\la(g)$. Then $\tau(f)=\hat f(e)$. Our first main result is the $L_p$ Poincar\'e type inequalities for $L_\8(\rz^d,\ga_d)\rtimes_{\hat\al} G$.

\begin{theorem}\label{gpp0}
  Let $G$ be a countable discrete group with an orthogonal representation $\al$ on a real Hilbert space $H$ of dimension $d\in\nz\cup \{\8\}$. Let $T_t$ act on $L_\8(\rz^d,\ga_d)\rtimes_{\hat\al} G$ as above, where $\hat\al$ is determined by $\al$. Then for $2\le p <\8$ and $f\in L_\8(\rz^d,\ga_d)\rtimes_{\hat\al} G$,  we have
  \begin{equation}\label{poi0}
      \|f-E_{\lx(G)} f\|_p\le C\sqrt{p} \max\{ \|\Ga_{L\rtimes I}(f,f)^{1/2}\|_p,\|\Ga_{L\rtimes I}(f^*,f^*)^{1/2}\|_p\}.
  \end{equation}
  Here $L\rtimes I$ is the generator of $T_t$.
\end{theorem}
We always understand $\|\Ga_A(f,f)^{1/2}\|_p=\8$ if $\Ga_A(f,f)$ is not well-defined. By density, it is easily seen that \eqref{poi0} still holds for unbounded $f\in L_p(L_\8(\rz^d,\ga_d)\rtimes_{\hat\al} G)$. To illustrate Theorem \ref{gpp0} for readers from more classical probability background, let us consider $G$ to be a finite group with $|G|=d$. Let $\al$ be an orthogonal representation of $G$ on $\rz^d$. Then $L_\8(\rz^d,\ga_d)\rtimes_{\hat\al} G\subset L_\8(\rz^d,\ga_d)\overline\otimes M_d=L_\8(\rz^d,\ga_d; M_d)$, where $M_d$ denotes the algebra of complex $d\times d$ matrices. Hence, we have the $L_p$ Poincar\'e inequalities for matrix-valued Gaussian functions in a certain subspace of $L_p(\rz^d,\ga_d; M_d)$. The proof is harder than the scalar-valued case because of the noncommutativity. This result will be proved in Section 3.

We may also consider certain Poincar\'e type inequalities for $\lx(G)$. Let $\psi$ be a conditional negative length (cn-length for short) function on $G$. Then it is well known that $\psi$ determines an orthogonal representation $\al$ on a real Hilbert space $H$ and a map $b_\psi:G\to H$ satisfying the cocycle law $b_\psi(gh)=b_\psi(g)+\al_g(b_\psi(h))$. $b_\psi$ is called a 1-cocycle on $G$. From here we can construct the group measure space $L_\8(\rz^d,\ga_d)\rtimes_{\hat\al} G$. Consider the semigroup $S_t$ acting on $\lx(G)$ defined by $S_t\la(g) = e^{-t\psi(g)}\la(g)$
for $g\in G$. Then $(S_t)$ is a standard semigroup. See \cite{JZ} for a proof of this fact. It extends
to a strongly continuous semigroup of contractions on $L_2(\lx(G))$. The generator is
given by $A\la(g) = \psi (g)\la(g)$. The following result can also be regarded as Poincar\'e type inequalities for $\lx(G)$.
\begin{cor}
  Let $G$ be a discrete group with cn-length function $\psi$. Let $2\le p <\8$. Then for $f\in\lx(G)$,
  \[
  \|\pi(f)-S_{1/2}(f)\|_{L_p(L_\8(\rz^d,\ga_d)\rtimes_{\hat\al} G)}\le C\sqrt{p} \max\{ \|\Ga_\psi(f,f)^{1/2}\|_{L_p(\lx(G))}, \|\Ga_\psi(f^*,f^*)^{1/2}\|_{L_p(\lx(G))}\}.
  \]
  Here $\pi: \lx(G)\to L_\8(\rz^d,\ga_d)\rtimes_{\hat\al} G$ is a trace preserving $*$-homomorphism given by $\pi(\la(g))=e^{i\lge b_\psi(g),\cdot\rge }\rtimes \la(g)$ and $b_\psi$ is the 1-cocycle determined by $\psi$.
\end{cor}

Let us return to the general setting and consider the noncommutative probability space $(\nx,\tau)$. Let $\phi(t)=e^{t^2}-1$. It is well known that $\phi$ is a Young function. Define for $x\in \nx$,
\[
\|x\|_\phi = \inf \{c>0: \tau[\phi(|x|/c)]\le 1\}.
\]
Then the Orlicz space $L_\phi(\nx)$ is the completion of $\nx$ in this norm.
Recall that the entropy of a positive $\tau$-measurable operator $\rho$ (see, e.g., \cite{FK}) is defined as
$$\Ent(\rho)=\tau\Big[\rho\ln \Big(\frac{\rho}{\tau(\rho)}\Big)\Big].$$
For a $\tau$-measurable operator $x$, we introduce the exponential integrability condition of Bobkov and G\"otze
\begin{equation}\label{exp0}
\tau(e^{x-E_{\Fix}x})\le \tau(e^{c\Ga_A(x,x)}).
\end{equation}
Then we have the following transportation type inequality.
\begin{theorem}\label{ttra}
  Let $(\nx,\tau)$ be a noncommutative probability space. Suppose the Poincar\'e type inequalities \eqref{poin} or the exponential integrability \eqref{exp0} hold for all self-adjoint $x\in\nx$ with $E_{\Fix}x=0$ and $\|\Ga_A(x,x)^{1/2}\|_{\phi}\le 1$. Then
  \begin{equation}\label{teq1}
    \sup_{ \|\Ga_A(x,x)^{1/2}\|_{\phi}\le 1} |\tau(x\rho)-\tau(xE_{\Fix}\rho)|\le C\max\{\sqrt{\Ent(\rho)}, ~ \Ent(\rho)\}
  \end{equation}
  for all $\tau$-measurable positive operator $\rho$ with $\tau(\rho)=1$.
\end{theorem}
Let us now indicate the connection between Theorem \ref{ttra} and transportation cost inequalities in classical probability. Let $(\Om, d)$ be a metric space equipped with a probability measure $\mu$. Assume $\nu$ is a probability measure absolutely continuous with respect to $\mu$. Suppose that there exists $x_0\in \Om$ such that $\int d(x,x_0) d\pz<\8$ for $\pz=\mu$ and $\nu$.  Let $g=\frac{d\nu}{d\mu}$. By the Kantorovich--Rubinstein formula (see, e.g., \cite{Vil}), the 1-Wasserstein distance can be written as
\begin{equation}\label{w1de}
W_1(\mu,\nu)=\sup_{\|f\|_{\Lip}\le 1}\Big|\int f gd\mu-\int f d\mu\Big|.
\end{equation}
Here $\|\cdot\|_{\Lip}$ denotes the Lipschitz constant. Suppose $\int e^{tf}d\mu\le e^{ct^2/2}$ for  all $t>0$ and all $f$ with $\int f d\mu =0$, $\|f\|_{\Lip}\le 1$. Then Bobkov and G\"otze showed in \cite{BG} that
\begin{equation}\label{bg1}
  W_1(\mu,\nu)\le \sqrt{2c \Ent(g)}=\sqrt{2c D(\nu||\mu)},
\end{equation}
for all $\nu$ absolutely continuous with respect to $\mu$. Here $D(\nu||\mu)=\int \ln \frac{d\nu}{d\mu}d\nu$ is the relative entropy; see also, e.g., \cites{Eff, Tro} and their references therein for the notion of quantum relative entropy.

Recall that the gradient form for Laplacian is the modulus of gradient $\Ga_{-\Delta}(f,f)=|\nabla f|^2 \le \|f\|_{\Lip}^2$. Given two states $\xi,\eta$ of the von Neumann algebra $\nx$, let us define
\begin{equation}\label{q1de}
Q_1(\xi, \eta) = \sup \{|\xi(x)-\eta(x)|: x \mbox{ self-adjoint}, \|\Ga_A(x,x)^{1/2}\|_{\8}\le 1\}.
\end{equation}
In the case that $\xi(\cdot)=\mu(\cdot)$, $\eta(\cdot)=\nu(\cdot)$ and the gradient form is associated to the Laplacian, we clearly have $W_1(\mu,\nu)= Q_1(\mu,\nu)$. The classical definition of Lipschitz functions is hard to generalize to the noncommutative setting. But the gradient form is well-defined. Therefore, Junge and the author simply take \eqref{q1de} as the definition of noncommutative 1-Wasserstein distance in \cite{JZ}. By extending the proof of \eqref{bg1}, in the same paper it was showed that if
\begin{equation}\label{exp1}
\tau(e^{t(x-E_{\Fix}x)})\le e^{ct^2}
\end{equation}
for all $t>0$ and self-adjoint $x\in \nx$ with $\|\Ga_A(x,x)^{1/2}\|_\8\le 1$, then
\begin{equation}\label{q1tr}
  Q_1(\rho,E_{\Fix}\rho)\le C\sqrt{\Ent(\rho)}
\end{equation}
for all $\tau$-measurable positive operators $\rho$ with $\tau(\rho)=1$. Here given two $\tau$-measurable operators $\rho,\si$, $Q_1(\rho,\si):=Q_1(\xi,\eta)$ for $\xi(x)=\tau(x\rho)/\tau(\rho)$ and $\eta(x)=\tau(x\si)/\tau(\si)$.  This quantity $Q_1$ is also closely related to the distance used by Rieffel to define his quantum metric spaces \cites{Ri1, Ri2}. Indeed, new quantum metric spaces were found using a variant of $Q_1$ (the self-adjoint condition was removed) in \cites{JMe,JMP}.

Notice that in \eqref{poi0} and \eqref{exp0}, we actually have better bounds than \eqref{exp1}. Hence, instead of requiring the strong condition $\|\Ga_A(x,x)^{1/2}\|_\8\le1$, we can ask for $\|\Ga_A(x,x)^{1/2}\|_\phi\le 1$. This motivates the following definition \begin{equation}\label{qphi}
   Q_{\phi}(\xi,\eta)=\sup \{|\xi(x)-\eta(x)|: x \mbox{ self-adjoint}, \|\Ga_A(x,x)^{1/2}\|_{\phi}\le 1\}.
\end{equation}
Then the conclusion \eqref{teq1} of Theorem \ref{ttra} can be rewritten as
\begin{equation}\label{teq2}
  Q_{\phi}(\rho, E_{\Fix}\rho)\le C\max\{\sqrt{\Ent(\rho)}, ~ \Ent(\rho)\}.
\end{equation}
Here $Q_\phi(\rho, E_{\Fix}\rho)$ is defined from $Q_{\phi}(\xi,\eta)$ in a similar way to $Q_1$. Notice that $\|\Ga_A(x,x)^{1/2}\|_{\phi}\le 1$ allows $\|\Ga_A(x,x)^{1/2}\|_p\le c\sqrt{p}$. Clearly, $Q_{\phi}$ is much bigger than $Q_1$ in general. Moreover, \eqref{teq2} implies that a phase transition behavior may happen for $Q_\phi$ depending on the entropy functional. In Section 4, we will prove Theorem \ref{ttra} and elaborate on its relationship with different transportation cost inequalities obtained in \cites{Ta,OV,BGL}. We also show that the linear term $\Ent(\rho)$ gives the correct order when the entropy is large. Thus \eqref{teq1} is sharp up to a constant. As an immediate consequence of Theorem \ref{ttra}, we show that the entropy functional gives an upper bound for the Wiener and Rademacher chaos of order 1 and 2. See Section 4 for the precise definition of chaos.
\begin{cor}
  Let $(\Om, \gx, \pz)$ be the Gaussian (resp. Rademacher) measure space. Let $\hx_n$ denote the the Wiener (resp. Rademacher) chaos of order $n$ and $P_{\hx_n}: L_2(\Om,\gx, \pz)\to \hx_n$  the orthogonal projection. Then for any positive function $f\in L_2(\Om,\gx,\pz)$ with $\|f\|_1=1$, we have
\[
\|P_{\hx_1}f\|_2\le C\sqrt{\Ent(f)},  \quad \|P_{\hx_2}(f)\|_2\le C'\max\{\sqrt{\Ent(f)},~\Ent(f)\}.
\]
\end{cor}
One can also deduce concentration and isoperimetric-type inequalities from the weaker inequality \eqref{q1tr} as indicated in \cite{BG}. The point here is that in the absence of logarithmic Sobolev inequality, \eqref{teq1} and \eqref{q1tr} may be good alternatives of transportation cost inequalities derived from LSI.

Furthermore, we will show that $\|\cdot\|_\phi$ gives new examples of quantum metric spaces in Section 5. Let $C_r^*(G)$ (resp. $\cz(G)$) denote the reduced group $C^*$-algebra of $G$ (resp. group algebra of $G$). Define $\opnorm{a}=\max\{\|\Ga_A(a,a)^{1/2}\|_\phi, \|\Ga_A(a^*,a^*)^{1/2}\|_\phi\}$ for $a\in \cz(G)$. See Section 5 for unexplained notation and terminology in the following result.
\begin{cor}
  $(C_r^*(G),\cz(G),\opnorm{\cdot})$ is a compact quantum metric space provided one of the following holds:
  \begin{enumerate}
    \item $G$ is finitely generated with rapid decay and $\inf_{|g|=k}\psi(g)\ge c_\al(1+k)^\al$ for some $\al>0$.
    \item ${\rm dim} H_\psi <\8$, the kernel of $\psi$ is $\{e\}$ and $\inf_{b_\psi(g)\neq 0}\psi(g)>0$. Here $b_\psi: G\to H_\psi$ is the 1-cocycle associated to $\psi$ on $G$.
  \end{enumerate}
\end{cor}


\section{Preliminaries}\label{prel}
From now on, we always assume $2\le p<\8$ unless we specify otherwise.
\subsection{Poincar\'e type inequalities for Gaussian measures}
Let $-L=\Delta-x\cdot \nabla$ be the generator of Ornstein--Uhlenbeck semigroup $P_t$ in $\rz^d$ for $d<\8$. Let $\ga_d$ denote the standard Gaussian measure on $\rz^d$. Then by e.g. \cite{Pi} the Mehler formula for $\cos^L\ta:=P_{-\ln\cos\ta}$ holds: for all $f\in L_2(\ga_d)$ and $0\le \ta\le \pi/2$
\begin{equation}
  \label{mehl}
  (\cos^L\ta f) (x)= \int_{\rz^d} f(x \cos \ta + y\sin \ta) \ga_d(d y).
\end{equation}
Following \cite{ELP}, we have the Poincar\'e type inequality for Gaussian measures, which should be a classical result; see \cite{JZ} for another proof based on martingale inequalities. We present the proof here because it is our guideline for the group measure space setting. We write $|\nabla f|=(\sum_{i=1}^d(\frac{\partial f}{\partial x_i})^2)^{1/2}$
\begin{prop}\label{clpo}
  Let $p\ge2$ and $\phi\in L_1([0,\pi/2])$. Then for all $f\in L_\8(\rz^d,\ga_d)$,
  \[
  \left\|\int_0^{\pi/2} \phi(\ta)\frac{\partial}{\partial\ta}\cos^L\ta (f) d\ta \right\|_{L_p(\ga_d)}\le C\sqrt{p} \|\phi\|_{L_1([0,\pi/2])}\||\nabla f|\|_{L_p(\ga_d)}.
  \]
  In particular,
  \[
  \|f-\int f d\ga_d\|_{L_p(\ga_d)}\le C\sqrt{p}\||\nabla f|\|_{L_p(\ga_d)}.
  \]
\end{prop}
\begin{proof}
  By approximation, we may assume that $f$ is a bounded $C^1$ function with  $|\nabla f|$ bounded by some polynomial so that we can differentiate under integral in the following. Since $\lim_{t\to \8} P_t f=\int f d\ga_d$, by \eqref{mehl}, we have
  \begin{align*}
  \int_0^{\pi/2} \phi(\ta)\frac{\partial \cos^L\ta (f)}{\partial \ta}(x) d\ta &=
  \int_0^{\pi/2} \phi(\ta)\frac{\partial}{\partial\ta}\int f(x \cos \ta + y\sin \ta) \ga_d(d y) d \ta\\
  &=\int_0^{\pi/2}\int_{\rz^d} \phi(\ta)\rx_\ta (\nabla f(x)\cdot y) \ga_d(dy) d\ta,
  \end{align*}
  where $\rx_\ta$ is a measure preserving automorphism of $(\rz^d\times \rz^d,\ga_d\times \ga_d)$ given by $(\rx_\ta F)(x,y)=F(x\cos\ta+y\sin\ta, -x\sin\ta+y\cos\ta)$. By Minkowski's integral inequality and H\"older's inequality,
  \begin{align*}
  &\left\|\int_0^{\pi/2} \phi(\ta)\frac{\partial}{\partial\ta}\cos^L\ta (f) d\ta \right\|_{L_p(\ga_d)} \\
\le & \int_0^{\pi/2}|\phi(\ta)|\int_{\rz^d}\left(\int_{\rz^d}|\rx_\ta(\nabla f(x)\cdot y)|^p\ga_d(dx)\right)^{1/p}\ga_d(dy)d\ta\\
  \le & \int_0^{\pi/2}|\phi(\ta)|\int_{\rz^d}\left(\int_{\rz^d}|\rx_\ta(\nabla f(x)\cdot y)|^p\ga_d(dx)\ga_d(dy)\right)^{1/p}d\ta\\
  \le &\|\phi\|_{L_1([0,\pi/2])}\left\|\sum_{i=1}^d \frac{\partial f}{\partial x_i}(x)y_i\right\|_{L_p(\ga_d\times\ga_d)}.
  \end{align*}
  The first assertion follows from Khintchine's inequality. Taking $\phi(\ta)=1_{[0,\pi/2]}(\ta)$ gives the second one as in \cite{ELP}.
\end{proof}
We remark that by approximation the above result also holds for the standard Gaussian measure on $\rz^\8$, because we still have the Mehler formula in this setting; see, e.g., \cite{Nua}.

\subsection{Crossed products}\label{crpr}
We briefly recall the crossed product construction. Our reference is  \cite{Tak, JMP}. Let $G$ be a discrete group with left regular representation $\la: G\to B(\ell_2(G))$. Given a noncommutative probability space $(\nx,\tau)$, we may assume $\nx\subset B(H)$ for some Hilbert space $H$. Suppose a trace preserving action $\al$ of $G$ on $\nx$ is given, i.e., we have a group homomorphism $\al: G\to \Aut(\nx)$ (the $*$-automorphism groups of $\nx$) with $\tau(x)=\tau(\al_g(x))$ for all $x\in\nx, g\in G$. Identify $\ell_2(G)\otimes H$ with $\ell_2(G; H)$. Consider the representation $\pi$ of $\nx$ on $\ell_2(G;H)$ given by
\[
\pi(x)=\sum_{g\in G}\al_{g^{-1}}(x)\otimes e_{g,g},
\]
where $e_{g,h}$ is the matrix unit of $B(\ell_2(G))$. In other words, $\pi(x)\xi(g)=\al_{g^{-1}}(x)\xi(g)$ for $x\in \nx, \xi\in \ell_2(G;H)$. Then the crossed product of $\nx$ by $G$, denoted by $\nx\rtimes_\al G$, is defined as the weak operator closure of $1_\nx\otimes \la(G)$ and $\pi(\nx)$ in $B(\ell_2(G;H))$. We usually drop the subscript $\al$ if there is no ambiguity. Clearly, $\nx\rtimes G$ is a von Neumann subalgebra of $\nx\overline \otimes B(\ell_2(G))$. In the special case $\nx=\cz$, the complex number algebra, $\cz\rtimes G$ reduces to the group von Neumann algebra $\lx(G)$. Therefore, $\lx(G)$ is a von Neumann subalgebra of $\nx\rtimes G$ and there exists a unique conditional expectation $E_{\lx(G)}:\nx\rtimes G\to \lx(G)$. A generic element of $\nx\rtimes G$ can be written as
\begin{align*}
  \sum_{g\in G}f_g\rtimes \la(g)&=\sum_{g\in G}\pi(f_g)\la(g) = \sum_{g,h,h'} (\al_{h^{-1}}(f_g)\otimes e_{h,h})(1_\nx\otimes e_{gh',h'})\\
  &=\sum_{g,h}\al_{h^{-1}}(f_g)\otimes e_{h,g^{-1}h}.
\end{align*}
There is a canonical trace on $\nx\rtimes G$ given by
\[
\tau\rtimes\tau_G(f\rtimes \la(g))=\tau\otimes \tau_G(f\otimes\la(g))=\tau(f)\de_{g=e},
 \]
where we denote by $\tau_G$ the canonical trace on $\lx(G)$. The arithmetic in $\nx\rtimes G$ is given by
\[
(f\rtimes \la(g))^*=\al_{g^{-1}}(f^*)\rtimes \la(g^{-1})
 \]
and
\[
(f_1\rtimes \la(g_1))(f_2\rtimes \la(g_2))=(f_1\al_{g_1}(f_2))\rtimes \la(g_1g_2).
\]
In what follows, we may simply write $f\la(g)$ instead of $f\rtimes \la(g)$. The group measure space in this paper refers to a special case of the crossed product, i.e., $\nx=L_\8(\Om, \mu)$ for some (standard) probability space $(\Om, \mu)$ .

\subsection{Gaussian measure space construction}
Our reference of this subsection is \cite{Str}*{Chapter 8} and \cite{Nua}*{Chapter 1}. Let $H$ be a real Hilbert space of dimension $d$, where $d\in\nz\cup\{+\8\}$. Identify $H$ as $\ell_2(d)$. Following the well known Gaussian measure space construction (see, e.g., \cites{Val}), we consider the linear map $B: \ell_2(d)\to L_2(\rz^d,\ga_d)$ given by $B(h)(y)=\sum_{i=1}^d\lge h,e_i\rge y_i$, where $(e_i)$ is an orthonormal basis of $\ell_2(d)$ and $y_i$ is the $i$-th coordinate map. If $d<\8$, $B(h)(y)=\lge h,y\rge$; if $d=\8$, $(\rz^d, \ga_d)$ is the measure space obtained from Kolmogorov's construction for which all the cylinder set measures are standard Gaussian measures. Note that $\lge B(h), B(k)\rge_{L_2(\rz^d,\ga_d)} =\lge h,k\rge_{\ell_2(d)}$. Let $\al: G\to O(H)$ be an orthogonal representation of $G$ on $H$. Then there exists a $G$-action $\al^*$ on $(\rz^d,\ga_d)$ preserving the Gaussian measure $\ga_d$; see \cite{Str}*{Theorem 8.3.14}. By abuse of notation, we simply write $\al$ for $\al^*$ because they are indeed the same if $d<\8$. The action $\al$ on $(\rz^d,\ga_d)$ induces an action $\hat{\al}$ on $L_2(\rz^d,\ga_d)$, such that $\hat{\al}_g(B(h))=B(\al_g(h))$ and
\begin{equation}\label{eqv0}
\hat{\al}_g(f)(x)=f(\al(g^{-1})x) = f(\al_{g^{-1}}(x) )
\end{equation}
for $f\in L_2(\rz^d,\ga_d)$. Clearly, $\hat\al$ extends naturally to isometric actions on $L_p(\rz^d,\ga_d)$ for $1\le p\le \8$. In the following we will consider the von Neumann algebra $\mx = L_\8(\rz^d, \ga_d)\rtimes_{\hat\al} G$ and simply forget the subscript $\hat\al$ in the notation of $\mx$ if there is no ambiguity.

Let $P_t$ be the Ornstein--Uhlenbeck semigroup acting on $L_\8(\rz^d,\ga_d)$; see \cites{Fan, Nua} for the case $d=\8$. Then $P_t\otimes id_{\ell_2(G)}$ is a semigroup acting on $L_\8(\rz^d,\ga_d) \overline\otimes B(\ell_2(G))$. Since the action $\al: G\curvearrowright (\rz^d,\ga_d)$ is linear and measure preserving, by the Mehler formula or by the functoriality of the Gaussian functor $\Ga$ \cite{BKS}, $P_t$ is $G$-equivariant, i.e.,
\begin{equation}\label{eqva}
  P_t\circ \hat\al_g=\hat\al_g\circ P_t.
\end{equation}
This will be the starting point of the $L_p$ Poincar\'e inequalities \eqref{poin} for group measure spaces in the next section, because our extension of the Ornstein--Uhlenbeck semigroups to the group measure spaces relies on \eqref{eqva}. Indeed, \eqref{eqva} implies
$$P_t\otimes id_{\ell_2(G)}(L_\8(\rz^d,\ga_d)\rtimes G )\subset L_\8(\rz^d,\ga_d)\rtimes G.$$
Define $P_t\rtimes id_{G}=P_t\otimes id_{\ell_2(G)}|_{L_\8(\rz^d,\ga_d)\rtimes G}$ and write $T_t=P_t\rtimes id_{G}$. Since the fixed point algebra of $P_t$ is trivial, the fixed point algebra of $T_t$ is $\lx(G)$. It is well known that $(T_t)$ extends to contractions on $L_p(L_\8(\rz^d,\ga_d)\rtimes G)$ and $\lim_{t\to \8} T_t f=E_{\lx(G)}f$ for $f\in L_p(L_\8(\rz^d,\ga_d)\rtimes G)$.

\subsection{Groups with affine representations}
Let $G$ be a countable discrete group with the conditional negative length (cn-length) function $\psi: G\to\rz_+$. Recall that $\psi$ is conditional negative if $\sum_g a_g=0\Rightarrow \sum_{g,h}\bar a_g a_h\psi(g^{-1}h)\le 0$. Then $\psi$ determines an affine representation which is given by an orthogonal representation $\al:G\to O(H)$ over a real Hilbert space $H$ together with a map $b_\psi:G\to H$ satisfying the cocycle law, i.e., $b_\psi(gh)=b_\psi(g)+\al_g(b_\psi(h))$; see, e.g., \cite{BO}. By the above Gaussian measure space construction, we get the finite von Neumann algebra $L_\8(\rz^d, \ga_d)\rtimes G$. Define the Gaussian derivation
\begin{align*}
  \de_\psi: \lx(G)&\to M_\8:=\cap_{0<p<\8}L_p(L_\8(\rz^d,\ga_d)\rtimes G),\\
  \la(g)&\mapsto B(b_\psi(g))\rtimes \la(g).
\end{align*}
Clearly, $\de_\psi$ is well-defined. Note that $M_\8$ is a $\lx(G)\dash\lx(G)$ bimodule with left and right actions given by $\la(h)(f\rtimes \la(g))=f\rtimes\la(hg)$ and $(f\rtimes \la(g))\la(h)=f\rtimes \la(gh)$. Then the derivation property $\de_\psi(f_1f_2)= f_1\de_\psi(f_2)+\de_\psi(f_1)f_2$ can be checked directly from the arithmetic in $L_\8(\rz^d,\ga_d)\rtimes G$. For our later development, we need to construct $b_\psi$ explicitly.
\begin{lemma}\label{fini}
For any $g,h\in G$, $b_\psi(g)\in \ell_2(d)$ and $\al_g(b_{\psi}(h))$ have at most finitely many nonzero coordinates.
\end{lemma}
\begin{proof}
Let $\rz G$ be the algebraic group algebra of $G$, i.e.,
\[
\rz G=\{x: x=\sum_{g} c_g\de_g, c_g\in \rz\}
\]
where the sum is over finite many elements. Let $K(g,h)=\frac12(\psi(g)+\psi(h)-\psi(g^{-1}h))$ for $g,h\in G$ and define
\[
 [\sum_g a_g \de_g, \sum_{g'}a_{g'}\de_{g'}]=\sum_{g,g'}a_g a_{g'} K(g,g').
\]
Since $\psi$ is conditional negative, $K$ is a positive semidefinite matrix. Put $N_\psi=\{x\in \rz G: [ x,x] =0\}$. Then we define an inner product on $\rz G/N_\psi$ as $\lge x+N_\psi,y+N_\psi\rge=[x,y]$ for $x,y\in\rz G$. Clearly $\lge\cdot,\cdot\rge$ is a well-defined inner product. Let $H_\psi$ be the norm closure of $\rz G/N_\psi$. We define $b_\psi: G\to H_\psi$ by $b_\psi(g)=\de_g+N_\psi$ and $\al_g(b_{\psi}(h))=b_\psi(gh)-b_\psi(g)$. Then $H_\psi, b_\psi, \al$ thus constructed satisfy the cocycle law. We may utilize the Gram--Schmidt procedure on $(b_\psi(g))_{g\in G}$ and obtain an orthonormal basis $(e_j)$ such that
\[
b_\psi(g_k)=\sum_{j=1}^k b_{kj} e_j,
\]
where $(g_k)$ is an enumeration of $G$. Hence, $b_\psi(g)$ only depends on finitely many $e_j$'s for all $g\in G$ and $H_\psi\cong \ell_2(d)$.
\end{proof}
A direct consequence of this construction is $\|b_\psi(g)\|^2=\psi(g)$ for $g\in G$.

\subsection{Khintchine inequality}
We briefly recall the modified Khintchine inequality derived in \cite{JMP}*{Section 4.1}. Let $(\mx,\tau)$ be a noncommutative probability space. Suppose a discrete group $G$ acts on $\mx$ and preserves the trace $\tau$. By the Gaussian measure space construction explained previously, we may consider the linear map $B:H\to L_2(\Om,\mu)$ given by $B(h)=\sum_k\lge h,e_k\rge \zeta_k$, where $(\zeta_k)$ is a family of centered independent Gaussian random variables in a probability space $(\Om,\mu)$, and $(e_k)$ is an orthonormal basis of $H$. Put
\[
G_p(\mx)\rtimes G=\{\sum_{h\in H}\sum_{g\in G} (B(h)\otimes f_{g,h})\la(g)\} \subset L_p(L_\8(\Om,\mu;\mx)\rtimes G),
\]
where $f_{g,h}$ is affiliated to $\mx$. Recall that conditional expectations between von Neumann algebras extend to contractions between noncommutative $L_p$ spaces. Consider the conditional expectation
\[
E: L_p(L_\8(\Om,\mu;\mx)\rtimes G)\to L_p(\mx\rtimes G)
, \quad E(\sum_g f_g\la(g))=\sum_{g}(\int_\Om f_g d\mu)\la(g).
\]
For $F\in L_\8(\Om,\mu;\mx)\rtimes G$, define the conditional row (resp. column) space $L_p^r(E)$ (resp. $L_p^c(E)$) with norm $\|F\|_{L_p^r(E)}=\|E(FF^*)^{1/2}\|_p$ (resp. $\|F\|_{L_p^c(E)}=\|E(F^*F)^{1/2}\|_p$). Put $L^{rc}_p(E)=L_p^c(E)\cap L_p^r(E)$. Define $RC_p(\mx)\rtimes G$ as $G_p(\mx)\rtimes G$ with the norm inherited from $L_p^{rc}(E)$. The following Khintchine inequality was proved in \cite{JMP}*{Theorem 4.3} with the best order of constant obtained in \cite{JZ1}.
\begin{theorem}\label{khin}
  Let $2\le p<\8$ and $F\in L_\8(\Om,\mu;\mx)\rtimes G$. Then
  \[
  \|F\|_{G_p(\mx)\rtimes G}\le C\sqrt{p} \|F\|_{RC_p(\mx)\rtimes G}.
  \]
\end{theorem}
We will use the case $\mx=L_\8(\rz^d,\ga_d)$ and $(\Om,\mu)=(\rz^d,\ga_d)$. Then we have $L_\8(\Om,\mu;\mx)\cong L_{\8}(\rz^{2d},\ga_{2d})$ and $L_{\8}(\rz^{2d},\ga_{2d})\rtimes G$ is simply extended from $L_{\8}(\rz^{d},\ga_{d})\rtimes G$ by diagonal action $\hat\al_g(\xi(\cdot)\eta(\cdot))(x,y) =(\hat\al_g\xi)(x)(\hat\al_g\eta)(y).$ It follows from Theorem \ref{khin} that
\begin{align*}
  &\Big\|\sum_{g}[\xi_g(x)\eta_g(y)]\la(g)\Big\|_{L_p(L_{\8}(\rz^{2d},\ga_{2d}) \rtimes G)}\\
\le &~ C\sqrt{p} \max\Big\{\Big\|\sum_{g,h}\Big(\int \hat\al_{g^{-1}}(\bar \xi_g(x)\xi_h(x))\ga_d(dx)\Big)[\hat\al_{g^{-1}}(\bar \eta_g(y)\eta_h(y))]\la(g^{-1}h)\Big\|^{1/2}_{L_{p/2}(L_{\8}(\rz^{d},\ga_{d}) \rtimes G)},\\
&\quad\Big\|\sum_{g,h}\Big(\int \hat\al_{g^{-1}}(\xi_g(x)\bar\xi_h(x)) \ga_d(dx)\Big)[ \hat\al_{g^{-1}}( \eta_g(y)\bar\eta_h(y)]\la(gh^{-1})\Big\|^{1/2}_{L_{p/2}(L_{\8}(\rz^{d},\ga_{d}) \rtimes G)}\Big\}.
\end{align*}
Here we assumed that $\xi$ is affiliated to  $L_\8(\Om,\mu)$ while $\eta$ is affiliated to $\mx$. Note that $\hat\al_g$ preserves the Gaussian measure $\ga_d$. In particular, if $\xi_g = \lge b_\psi(g),\cdot\rge$, then
\[
\int \bar \xi_g(x)\xi_h(x)\ga_d(dx)=\lge b_\psi(g), b_\psi(h)\rge_{\ell_2(d)}.
\]
Similarly, for $\mx=\cz$ and $f\in\lx(G)$, we have
\begin{equation}\label{khgp}
    \|\de_\psi f\|_{L_p(L_\8(\rz^d, \ga_d)\rtimes G)}\le C\sqrt{p}\max\{\|\Ga_\psi(f,f)^{1/2}\|_{L_p(\lx(G))}, \|\Ga_\psi(f^*,f^*)^{1/2}\|_{L_p(\lx(G))}\}.
\end{equation}
The right-hand side of \eqref{khgp} follows from the arithmetic of crossed products as explained in Section \ref{crpr}. A detailed calculation was given in the proof of \cite{JMP}*{Theorem 4.6}.

\subsection{Noncommutative Orlicz spaces}
Our reference of this subsection is \cite{FK}. Let $\mx$ be a semifinite von Neumann algebra with a normal semifinite faithful trace $\tau$. Given a $\tau$-measurable operator $x$, the distribution function of $x$ is defined as
\[
\la_s(x)=\tau(E_{(s,\8)}(|x|)),\quad s>0
\]
where $E_{(s,\8)(|x|)}$ is the spectral projection of $|x|$ corresponding to the interval $(t,\8)$. The generalized singular number of $x$ is given by
\[
\mu_t(x)=\inf\{s\ge0: \la_s(x)\le t\}.
\]
Then \cite{FK}*{Corollary 2.8} asserts that, for any continuous increasing function $f$ on $[0,\8)$ with $f(0)=0$ and any $\tau$-measurable operator $x$, one has
\begin{equation}\label{sing}
  \tau(f(|x|))=\int_0^\8 f(\mu_t(x)) dt.
\end{equation}
Recall that a Young (or Orlicz in some literature) function $\phi:\rz_+\to\bar\rz_+$ is convex, increasing  with $\phi(0)=0$ and $\lim_{t\to\8}\phi(t)=\8$. The noncommutative Orlicz space $L_\phi(\mx,\tau)$ is defined as the space of all $\tau$-measurable operators such that $\tau(\phi(|x|/c))<\8$ for some $c>0$. $L_\phi(\mx,\tau)$ is a Banach space with the norm
\[
\|x\|_\phi = \inf \{c>0: \tau[\phi(|x|/c)]\le 1\}.
\]
$L_\phi(\mx,\tau)$ can also be defined as a noncommutative symmetric function space; see, e.g., \cites{BC,Xu,PS,KS} and the references therein for more information. When $\tau$ is finite, $L_\phi(\mx,\tau)$ can be obtained from the completion of $\mx$ in this norm. In the following, we will mainly consider the Orlicz space with $\phi(t)=e^{t^2}-1$. The following fact is standard. We include a quick proof for completeness.
\begin{prop}\label{orli}
  Let $a,b, x, y$ be $\tau$-measurable operators. Then,
  \begin{enumerate}
    \item $\|x\|_\phi=\|x^*\|_\phi=\||x|\|_\phi$;
    \item if $0\le x\le y$, then $\|x\|_\phi\le \|y\|_\phi$.
  \end{enumerate}
  In particular, $\|ax\|_\phi\le \|a\|_\8\|x\|_\phi$, $\|x b\|_\phi\le \|x\|_\phi \|b\|_\8$.
\end{prop}
\begin{proof}
  Note that $\psi(x):=\phi(x/c)$ is a Young function. Since $\mu_t(x)=\mu_t(x^*)=\mu_t(|x|)$ (\cite{FK}*{Lemma 2.5}), by \eqref{sing}, we have $\tau(\psi(|x|))=\tau(\psi(|x^*|))$. Then the first assertion follows. If $0\le x\le y$, then $\mu_t(x)\le \mu_t(y)$ for all $t>0$ (\cite{FK}*{Lemma 2.5}). Since $\psi$ is increasing, using \eqref{sing} again, we find $\tau(\psi(x))\le \tau(\psi(y))$. This gives the second assertion. Notice that $0\le x^*a^*ax\le \|a\|_\8^2 x^*x$ and that the square root is operator monotone. We have $0\le |ax|\le \|a\|_\8 |x|$. It follows that $\|ax\|_\phi\le \|a\|_\8\|x\|_\phi$. The last inequality is immediate once we note that $\|x b\|_\phi=\|b^*x^*\|_\phi$.
\end{proof}

\section{Poincar\'e type inequalities }
\subsection{Group measure spaces}
The idea of our proof for the Poincar\'e type inequalities goes back to Pisier \cite{Pi} where he deduced a magic formula to connect the Riesz transform and the Gaussian measure space. This strategy was further developed by Lust-Piquard in various situations. In particular, in \cite{ELP} Efraim and Lust-Piquard proved the Poincar\'e type inequalities for Walsh systems and CAR algebras following her earlier works, which motivates our proof.
\begin{lemma}\label{eqv1}
Let $\hat\al: G\to \Aut(L_\8(\rz^d,\ga_d))$ be the measure preserving action given by \eqref{eqv0}. Suppose $f\in L_\8(\rz^d,\ga_d)$ is differentiable and depends on finitely many coordinates if $d=\8$. Then for $g\in G$,
\[
\frac{\partial \hat\al_g(f)}{\partial x_i} (x) = \lge(\nabla f)(\al_{g^{-1}}(x)), \al_{g^{-1}}(e_i)\rge,
\]
where $(e_i)$ is the standard basis of $\rz^d$. Therefore, $(\nabla\hat\al_g(f))(x)=\al_g[(\nabla f)(\al_{g^{-1}}(x))]$ and
\[
\lge \nabla \hat\al_g(f))(x), y\rge = \lge (\nabla f) (\al_{g^{-1}}(x)),\al_{g^{-1}}(y) \rge,
\]
where $\lge\cdot,\cdot\rge$ is the inner product in $\ell_2(d)$.
\end{lemma}
\begin{proof}
This is just the chain rule. Here is the direct calculation.
\begin{align*}
  &\frac{\partial }{\partial x_i} \hat\al_g(f)(x)=\lim_{t\to 0} \frac{\hat\al_g(f)(x+t e_i)-\hat\al_g(f)(x)}{t} \\
  =&\lim_{t\to 0} \frac{f(\al(g^{-1})(x)+t \al(g^{-1})(e_i))-f(\al(g^{-1})x)}{t}\\
  =&\lge (\nabla f)(\al(g^{-1})x), \al(g^{-1})e_i\rge=\lge \al(g)[(\nabla f)(\al(g^{-1})x)], e_i\rge.
\end{align*}
This gives the gradient of $\hat\al_g(f)$ at $x$.
\end{proof}
We follow the notation in Section 2. Let $f(x,y)$ be a measurable function on $(\rz^d\times\rz^d,\ga_d\times \ga_d)$. Recall that $\rx_\ta$ is the measure preserving automorphism on $L_\8(\rz^d\times\rz^d,\ga_d\times \ga_d)$ given by $(\rx_\ta f)(x,y)=f(x\cos\ta+y\sin\ta,-x\sin\ta+y\cos\ta)$, and that $T_t=P_t\rtimes id_G$ is the semigroup acting on $L_\8(\rz^d,\ga_d)\rtimes G$, which is a natural extension of the Ornstein--Uhlenbeck semigroup $P_t$ on $L_\8(\rz^d,\ga_d)$.
\begin{lemma}\label{diff}
  Suppose that $\xi$ is a bounded $C^1$ function on $\rz^d$ with $|\nabla \xi|$ bounded by some polynomial and $\xi$ depends on only finitely many coordinates. Then for $g\in G$ and $0<\ta<\pi/2$,
\begin{align*}
\frac{\partial}{\partial\ta}T_{-\ln\cos\ta} (\xi\rtimes \la(g)) (x) = \sum_h \int\rx_\ta[\lge (\nabla \xi)(\al_{h}(x)), \al_{h}(y)\rge] \ga_d(dy)\otimes e_{h,g^{-1}h} .
\end{align*}
\end{lemma}
\begin{proof}
Since the Mehler formula \eqref{mehl} can be extended naturally to $T_t$, we have
\begin{align*}
T_{-\ln\cos\ta} (\xi\rtimes \la(g)) (x)&= \sum_h P_{-\ln\cos\ta}((\hat{\al}_{h^{-1}}\xi))(x)\otimes e_{h,g^{-1}h}\\
&= \sum_h \int (\hat\al_{h^{-1}}\xi)(x \cos\ta+y\sin\ta)\ga_d(dy)\otimes e_{h,g^{-1}h}.
\end{align*}
By Lemma \ref{eqv1}, we have
\begin{align*}
\frac{\partial}{\partial\ta}(\hat\al_{h^{-1}}\xi)(x\cos\ta+y\sin\ta)&=\lge (\nabla\hat\al_{h^{-1}}(\xi))(x\cos\ta+y\sin\ta), -x\sin\ta+y\cos\ta\rge\\
&=\lge (\nabla \xi)(\al_{h}(x\cos\ta+y\sin\ta)), \al_{h}(-x\sin\ta+y\cos\ta)\rge\\
&=\rx_\ta(\lge (\nabla \xi)(\al_{h}(x)), \al_{h}(y)\rge).
\end{align*}
We have assumed $\xi$ to be a nice function so that we can differentiate entrywise under integral and find
\[
\frac{\partial}{\partial\ta}T_{-\ln\cos\ta} (\xi\rtimes \la(g)) (x) = \sum_h \int\rx_\ta[\lge (\nabla \xi)(\al_{h}(x)), \al_{h}(y)\rge] \ga_d(dy)\otimes e_{h,g^{-1}h} . \qedhere
\]
\end{proof}
Recall that the fixed point algebra of $T_t$ is $\lx(G)$. Let $L$ denote the generator of $P_t$ and $L\rtimes I$ the generator of $T_t$.
\begin{theorem}\label{gmpo}
  Let $2\le p<\8$ and $f\in L_\8(\rz^d,\ga_d)\rtimes_{\hat\al} G$ where the action $\hat\al$ is the measure preserving action determined by the orthogonal representation $\al$ given by \eqref{eqv0}. Then
  \begin{equation}
  \begin{split}
    &\|f-E_{\rm \lx(G)} f\|_{L_p(L_\8(\rz^{d},\ga_{d})\rtimes G)}\\
  \le&~ C\sqrt{p} \max\{\|\Ga_{L\rtimes I}(f,f)^{1/2}\|_{L_p(L_\8(\rz^d,\ga_d)\rtimes G)}, \|\Ga_{L\rtimes I}(f^*,f^*)^{1/2}\|_{L_p(L_\8(\rz^d,\ga_d)\rtimes G)} \}. 
    \end{split}\tag{\ref{poi0}}
  \end{equation}
\end{theorem}
\begin{proof}
We follow the strategy of Proposition \ref{clpo} and take advantage of the techniques developed in \cite{JMP}. By approximation, we may assume that $f=\sum_{g\in G} f_g\la(g)\in L_\8(\rz^d,\ga_d)\rtimes G$ for finitely many $g\in G$ and that $f_g$'s satisfy the assumption of Lemma \ref{diff}. Note that $L_\8(\rz^d,\ga_d)\rtimes G \subset L_\8(\rz^d,\ga_d)\overline\otimes B(\ell_2(G))$. Then we have
\[
T_{-\ln\cos(\pi/2-\eps)}f- f=\int_0^{\pi/2-\eps} \frac{\partial}{\partial\ta}T_{-\ln\cos\ta}f d\ta.
\]
Sending $\eps\to 0$, we have
\[
E_{\lx(G)}f-f=\int_0^{\pi/2} \frac{\partial}{\partial\ta}T_{-\ln\cos\ta}f d\ta.
\]
We can extend the action $\hat\al: G\to \Aut(L_\8(\rz^d,\ga_d))$ to $$G\curvearrowright L_\8(\rz^{2d},\ga_{2d})=L_\8(\rz^{d},\ga_{d})\overline\otimes L_\8(\rz^{d},\ga_{d})$$
by diagonal action
$$\hat\al_g(\xi(\cdot)\eta(\cdot))(x,y)=(\hat\al_g\xi)(x)(\hat\al_g\eta)(y).$$
Noticing that $\hat\al_{g}\circ \rx_\ta=\rx_\ta\circ\hat\al_{g}$ for $g\in G$ and $\ta\in [0,\pi/2]$, by Lemma \ref{diff}, we have
\begin{align}
E_{\lx(G)}f-f&=\int_0^{\pi/2} \sum_{g,h}\int  \rx_\ta[ \lge \nabla f_g (\al_h(x)), \al_{h}(y)\rge] \ga_d(dy)\otimes e_{h,g^{-1}h}d\ta\nonumber\\
&=\int_0^{\pi/2} \sum_{g,h}\int \hat\al_{h^{-1}}[\rx_\ta( \nabla f_g(x), y\rge)] \ga_d(dy)\otimes e_{h,g^{-1}h}d\ta\nonumber\\
&= \int_0^{\pi/2} E_{L_\8^x(\rz^d,\ga_d)\rtimes G} \Big[\sum_{g} \rx_\ta[\lge \nabla f_g(x), y\rge]\rtimes\la(g)\Big] d\ta\nonumber\\
&= \int_0^{\pi/2} E_{L_\8^x(\rz^d,\ga_d)\rtimes G}\Big[\sum_{g} (\rx_\ta\otimes id_{\ell_2(G)})[(\lge \nabla f_g(x),y\rge)\rtimes \la(g)]\Big] d\ta.\label{minu}
\end{align}
Here we used the facts that $\sum_{g,h} \hat\al_{h^{-1}}[\rx_\ta( \lge \nabla f_g (x), y\rge)]\otimes e_{h,g^{-1}h}$ is in $L_p(L_\8(\rz^{2d},\ga_d)\rtimes G)$ for $2\le p <\8$ and that the conditional expectation
\[
E_{L_\8^x(\rz^d,\ga_d)\rtimes G}: L_\8(\rz^{2d},\ga_{2d})\rtimes G\to L_\8^x(\rz^d,\ga_d)\rtimes G
\]
extends to a contraction on $L_p(L_\8(\rz^{2d},\ga_{2d})\rtimes G)$. It follows from \eqref{eqva} or the chain rule that $\hat\al_g\circ L= L\circ \hat\al_g$. By the arithmetic of crossed products as explained in Section \ref{crpr}, we have
\begin{align*}
  &\Ga_{L\rtimes I}(f,f)\\
  =& \frac12\sum_{g,h} [L(\hat\al_{g^{-1}}(\bar f_g))\hat \al_{g^{-1}}(f_h)+\hat\al_{g^{-1}}(\bar f_g)\hat\al_{g^{-1}}(Lf_h)-L(\hat\al_{g^{-1}}(\bar f_g)\hat\al_{g^{-1}}(f_h))]\la(g^{-1}h)\\
  =&\sum_{g,h}\hat\al_{g^{-1}}(\lge \nabla f_h, \nabla f_g\rge )\la(g^{-1}h) \\
  =&\sum_{g,h}\lge \nabla f_h(\al_g(\cdot)), \nabla f_g(\al_g(\cdot))\rge \rtimes \la(g^{-1}h).
\end{align*}
We write $E_x$ for $E_{L_\8^x(\rz^d,\ga_d)\rtimes G}$. By the definition of $L_p^c(E)$, we have
\begin{align}
  &\|\sum_g \lge \nabla f_g(x),y\rge \rtimes \la(g)\|_{L_p^c(E_x)}\nonumber\\ 
  &= \|\sum_{g,h} \int\hat\al_{g^{-1}}(\overline{\lge \nabla f_g(x),y\rge}\lge \nabla f_h(x),y\rge) \ga_d(dy) \rtimes \la(g^{-1}h)\|^{1/2}_{L_{p/2}(L_\8(\rz^d,\ga_d)\rtimes G)}\nonumber\\
  &= \|\sum_{g,h}\lge \nabla f_h(\al_g(x)), \nabla f_g(\al_g(x))\rge \rtimes \la(g^{-1}h)\|^{1/2}_{L_{p/2}(L_\8(\rz^d,\ga_d)\rtimes G)}\nonumber\\
  &=\|\Ga_{L\rtimes I}(f,f)^{1/2}\|_{L_{p}(L_\8(\rz^d,\ga_d)\rtimes G)}. \label{gax}
\end{align}
Similarly,
\[
\|\sum_g \lge \nabla f_g(x),y\rge \rtimes \la(g)\|_{L_p^r(E_x)} = \|\Ga_{L\rtimes I}(f^*,f^*)^{1/2}\|_{L_{p}(L_\8(\rz^d,\ga_d)\rtimes G)}.
\]
It follows from Theorem \ref{khin}, \eqref{minu} and \eqref{gax} that
\begin{align*}
  &\|f-E_{\lx(G)}f\|_{L_p(L_\8(\rz^d,\ga_d)\rtimes G)}\\
  \le &~ \frac{\pi}2 \|\sum_g \lge \nabla f_g(x),y\rge \rtimes \la(g)\|_{L_p(L_\8(\rz^d\times \rz^d, \ga_d\times \ga_d)\rtimes G)}\\
\le&~ \frac{C\pi}2\sqrt{p} \max\{\|\Ga_{L\rtimes I}(f,f)^{1/2} \|_{L_p(L_\8(\rz^d,\ga_d)\rtimes G)}, \|\Ga_{L\rtimes I}(f^*,f^*)^{1/2}\|_{L_p(L_\8(\rz^d,\ga_d)\rtimes G)} \}.
\end{align*}
This completes the proof.
\end{proof}
\begin{rem}
In an earlier version of this paper, Theorem \ref{gmpo} was stated in the context that the representation $\al$ is determined by a conditional negative length function on $G$. The current form was suggested by the referee. The orthogonal representation $\al$ is crucial in our argument. It seems it is not enough to only assume \eqref{eqva} for a general measure preserving action $\al$, because such an action may destroy the Gaussian structure and the differentiability of functions.
\end{rem}

\subsection{Group von Neumann algebras}
Recall that $S_t$ is the semigroup acting on $\lx(G)$ given by $S_t\la(g)=e^{-t\psi(g)}\la(g)$. Define
\[
\pi: \lx(G)\to L_\8(\rz^d,\ga_d)\rtimes G,\quad
\pi(\la(g))=e^{i\lge b_\psi(g),\cdot\rge }\rtimes \la(g),
 \]
where $b_\psi$ is the the 1-cocycle determined by $\psi$. Then $\pi$ is a trace preserving $*$-homomorphism.
\begin{cor}\label{gppo}
  Let $G$ be a discrete group with cn-length function $\psi$. Let $2\le p <\8$. Then for $f\in\lx(G)$,
  \[
  \|\pi(f)-S_{1/2}(f)\|_{L_p(L_\8(\rz^d,\ga_d)\rtimes G)}\le C\sqrt{p} \max\{ \|\Ga_\psi(f,f)^{1/2}\|_{L_p(\lx(G))},\|\Ga_\psi(f^*,f^*)^{1/2}\|_{L_p(\lx(G))}\}.
  \]
  \end{cor}
\begin{proof} The proof is similar to the proof of Theorem \ref{gmpo}. By approximation, it suffices to consider $f=\sum_{g} f_g\la(g)\in\lx(G)$ for finitely many $g$'s. Note that
\[
E_{\lx(G)}\pi(f)=\sum_g f_g \int e^{i\lge b_\psi(g),x\rge } \ga_d(dx) \la(g)=\sum_g f_g e^{-\psi(g)/2}\la(g) = S_{1/2} f.
\]
By \eqref{minu}, Lemma \ref{fini} and Lemma \ref{diff} with $\xi=f_g\exp(i\lge b_\psi(g),\cdot\rge)$, we have
\begin{align*}
S_{1/2} f-\pi(f)&= i\int_0^{\pi/2} E_{L_\8^x(\rz^d,\ga_d)\rtimes G}\Big[\sum_{g} f_g (\rx_\ta\otimes id_{\ell_2(G)})[( e^{i\lge b_\psi(g),x\rge} \lge b_\psi(g),y\rge)\rtimes \la(g)]\Big] d\ta.
\end{align*}
We extend the trace preserving $*$-homomorphism $\pi$
\begin{equation}
  \label{tph}
\begin{aligned}
  \pi: L_\8(\rz^{d}, \ga_{d})\rtimes G&\to L_\8(\rz^{2d}, \ga_{2d})\rtimes G\\
   f(y)\rtimes \la(g)&\mapsto [\exp(i\lge b_\psi(g),x \rge) f(y)] \rtimes \la(g).
\end{aligned}
\end{equation}
Clearly, $\pi$ extends to linear isometries between $L_p$ spaces so that the function $f$ in \eqref{tph} can be in $L_p(\rz^d,\ga_d)$.
Hence, we find the crucial identity
\[
S_{1/2}f-\pi(f)=iE_{L_\8^x(\rz^d,\ga_d)\rtimes G}\left[\int_0^{\pi/2}(\rx_\ta\otimes id_{\ell_2(G)}) (\pi^x \de_\psi^y(f))d\ta\right],
\]
where the $x,y$ in the superscript are used to specify the variables in order to take conditional expectation. It follows that
\[
\|\pi(f)- S_{1/2}f\|_{L_p(L_\8(\rz^{d},\ga_{d})\rtimes G)}\le \frac{\pi}2 \|\de_\psi f\|_{L_p(L_\8(\rz^{d},\ga_{d})\rtimes G)}.
\]
Then \eqref{khgp} completes the proof.
\end{proof}

\section{Transportation type inequalities and applications}
The readers are referred to \cite{Vil} for general questions on transportation problems. Let $(\Om, d)$ be a metric space. Let $\mu$ and $\nu$ be probability measures on $(\Om, d)$ with finite $p$-th moment. Recall that the $p$-Wasserstein distance between $\mu$ and $\nu$ is defined as
\[
W_p(\mu,\nu)=\inf \Big(\iint d(x,y)^p d\pi(x,y)\Big)^{1/p}
\]
where the infimum is taken over all probability measure $\pi$ on the product space $\Om\times \Om$ which is a coupling of $\mu$ and $\nu$. The $1$-Wasserstein distance has a functional representation as \eqref{w1de}.
There is also a functional representation of $W_2$ (see, e.g., \cite{Ra})
\begin{equation}\label{w2fu}
W_2^2(\mu,\nu)=\sup\int g d\nu-\int f d\mu
\end{equation}
where the sup is taken over all bounded continuous function pairs $(g,f)$ with $g(y)-f(x)\le d(x,y)^2$ for all $x,y\in\Om$.

The starting point of this section is the $L_p$ Poincar\'e type inequalities \eqref{poin}, i.e., for $2\le p<\8$,
\begin{equation}
      \|x-E_{\rm Fix} x\|_p\le C\sqrt{p} \max\{ \|\Ga_A(x,x)^{1/2}\|_p,\|\Ga_A(x^*,x^*)^{1/2}\|_p\}.\tag{\ref{poin}}
  \end{equation}
The known examples satisfying these inequalities include the classical Gaussian spaces as shown in Proposition \ref{clpo}, the Walsh systems and CAR algebras due to Efraim and Lust-Piquard \cite{ELP}, and the group (Gaussian) measure space proved in Theorem \ref{gmpo}. In a forthcoming paper, we will show that the free group von Neumann algebras and more examples satisfy \eqref{poin}. We recall a result from \cite{JZ}*{Corollary 3.19}, which is a noncommutative generalization of Bobkov and G\"otze's result in \cite{BG}. Let $(\nx,\tau)$ be a noncommutative probability space.
\begin{theorem}
Suppose $\tau(e^{t(x-E_{\rm Fix} x)})\le e^{ct^2}$ for any $\tau$-measurable self-adjoint operator $x$ affiliated to $\nx$ such that $\|\Ga_A(x,x)\|_\8\le 1$. Then
\begin{equation}\label{wse2}
 Q_1(\rho, E_{\rm Fix} \rho)\le \sqrt{2c \Ent(\rho)}.
\end{equation}
for all $\rho\ge0$ with $\tau(\rho)=1$.
\end{theorem}
The definition of $Q_1$ was given \eqref{q1de}. As explained there, in the commutative setting with Laplacian as the generator, $Q_1 = W_1$. With the help of Poincar\'e type inequalities \eqref{poin}, one can easily show that for $t\in\rz$ and all self-adjoint element $x\in \nx$,
\begin{equation}\label{expx}
  \tau(e^{t(x-E_{\Fix}x)})\le e^{c\|\Ga_A(x,x)\|_\8 t^2}
\end{equation}
for some constant $c>0$. In fact, we only need the $L_\8$ norm instead of the $L_p$ norm on the right-hand side of \eqref{poin} to deduce \eqref{expx}; see \cite{JZ}*{Proposition 3.14}. This means that the transportation type inequality \eqref{wse2} holds for all the examples mentioned above which satisfy the $L_p$ Poincar\'e inequalities \eqref{poin}.

Assuming logarithmic Sobolev inequality, Bobkov and G\"otze proved an exponential integrability result, which reads in our context as
\begin{equation}\label{exp2}
  \tau(e^{x-E_{\Fix} x})\le \tau(e^{c\Ga_A(x,x)}).\tag{\ref{exp0}}
\end{equation}
Note that \eqref{exp2} is stronger than \eqref{expx}. Our next result says that \eqref{exp2} can be derived from \eqref{poin}.
\begin{theorem}
  Suppose the $L_p$ Poincar\'e type inequalities \eqref{poin} hold for all $p\ge 2$ and all self-adjoint $x\in\nx$. Then there exists $c>0$ such that \eqref{exp2} holds for all self-adjoint $x\in \nx$.
\end{theorem}
\begin{proof}
  Since $E_{\Fix}x$ is in the multiplicative domain of $T_t$, we have $\Ga_A(x,x)=\Ga_A(x-E_{\Fix}x, x-E_{\Fix}x)$. Without loss of generality, we may assume $E_{\Fix}x=0$ and thus $\tau(x)=0$. By the Taylor series, we have
  \begin{align*}
  &\tau(e^{x})= 1+\sum_{k=2}^\8 \frac{\tau(x^k)}{k!} \le 1+\sum_{k=2}^\8 \frac{C^k k^{k/2}\tau(\Ga_A(x,x)^{k/2})}{k!}\\
  &=1+\sum_{j=1}^{\8} \frac{C^{2j} (2j)^{j}\tau(\Ga_A(x,x)^{j})}{(2j)!}+\sum_{j=1}^{\8} \frac{C^{2j+1} (2j+1)^{j+1/2}\tau(\Ga_A(x,x)^{j+1/2})}{(2j+1)!}.
  \end{align*}
  Choose $\ta\in(0,1)$ so that $\frac1{j+1/2}=\frac{1-\ta}{j}+\frac{\ta}{j+1}$. By the noncommutative H\"older inequality,
  \[
  \|\Ga_A(x,x)\|_{j+1/2}\le \|\Ga_A(x,x)\|_{j}^{1-\ta}\|\Ga_A(x,x)\|_{j+1}^\ta.
  \]
  By Young's inequality, we have
  \begin{align*}
      \tau(\Ga_A(x,x)^{j+1/2})&\le \frac{(2j+1)(1-\ta)}{2j}\tau(\Ga_A(x,x)^{j})+ \frac{(2j+1)\ta}{2j+2}\tau(\Ga_A(x,x)^{j+1})\\
      &\le \max\{\tau(\Ga_A(x,x)^{j}),~ \tau(\Ga_A(x,x)^{j+1})\}.
  \end{align*}
  Note that for $j\ge 1$,
  \[
  \frac{(2j+1)^{j+1/2}}{(2j+1)!}\le \min\Big\{\frac{(2j)^j}{(2j)!}, \sqrt{2j+2}\frac{(2j+2)^{j+1}}{(2j+2)!}\Big\}.
  \]
  Since $\sqrt{2j+2}$ is bounded by $C^{2j+2}$ for some $C>1$, we have
  \begin{align*}
  \tau(e^x)&\le 1+2\sum_{j=1}^{\8} \frac{C'^{2j} (2j)^{j}\tau(\Ga_A(x,x)^{j})}{(2j)!}+\sum_{j=2}^{\8} \frac{C'^{2j} (2j)^{j}\tau(\Ga_A(x,x)^{j})}{(2j)!}\\
  &\le 1+3\sum_{j=1}^{\8} \frac{C'^{2j} (2j)^{j}\tau(\Ga_A(x,x)^{j})}{(2j)!}.
  \end{align*}
  Notice the elementary inequality $\frac{(2j)^j}{(2j)!}=\frac{j^j}{j!(2j-1)!!}\le (\frac{e}2)^j\frac1{j!}$ for $j\in \nz$. We have
  \[
  \tau(e^x)\le 1+ \sum_{j=1}^{\8} \frac{C^{2j} (e/2)^j\tau(\Ga_A(x,x)^{j})}{j!}=\tau(e^{c\Ga_A(x,x)}).\qedhere
  \]
\end{proof}

Now that the better bounds in \eqref{poin} (compared with (3.2) in \cite{JZ}) result in the stronger exponential inequality \eqref{exp2} (compared with \eqref{expx}), we expect to have stronger transportation type inequality as well. On the other hand, Bobkov and G\"otze actually showed a sharper inequality based on logarithmic Sobolev inequality in the classical setting in \cite{BG}
\begin{equation}\label{shap}
Q(\mu,\nu):=\sup_{\|f\|_{\rm Lip}\le 1}\|f\|_{L^2(d\nu)}-\|f\|_{L^2(d\mu)}\le \sqrt{2c D(\nu||\mu)}.
\end{equation}
It was observed in the same paper that $W_1(\mu,\nu)\le Q(\mu,\nu)\le W_2(\mu,\nu)$. Our goal here is to give a stronger version (compared with \eqref{wse2}) of the transportation type inequality in the spirit of \eqref{shap}. Since we have all the $L_p$ norm of $\Ga_A(x,x)^{1/2}$ in the Poincar\'e type inequalities, it is natural to consider the situation where we do not have $\|\Ga_A(x,x)^{1/2}\|_\8\le 1$ but some mild control on $\|\Ga_A(x,x)^{1/2}\|_p$. Our approach is related to Milman's generalization \cite{Mi} of transportation cost inequalities in the commutative setting.

Given a Young function $\phi$, the complementary function $\phi^*$ is given by the Legendre transform $\phi^*(s)=\sup_{t\ge0}\{ ts-\phi(t)\}$ for $s\ge 0$. The following lemma follows easily from Young's inequality. We include a proof for completeness following \cites{BG, Mi}. Recall that $\Ent(y)=\tau(y\ln(y/\ln(y)))$ for a positive $\tau$-measurable operator $y$. We will need the exponential integrability
\begin{equation}\label{exp}
\tau(e^{t(x-E_{\Fix}x)})\le e^{\varphi(t)}.
\end{equation}

\begin{lemma}\label{ent1}
  Let $\varphi:[0,\8)\to[0,\8]$ be a strictly increasing Young function. Assume \eqref{exp} holds for all $t\ge0$ and all self-adjoint $\tau$-measurable operator $x$. Then for all positive operator $\rho\in \nx$ with $\tau(\rho)=1$, we have for any self-adjoint $x$,
  \[
  \tau(x\rho-xE_{\rm Fix}\rho)\le (\varphi^*)^{-1}(\Ent(\rho)).
  \]
  Here $h^{-1}:[0,\8)\to[0,\8]$ is the inverse of the function $h$ given by $$h^{-1}(t)=\sup\{s: h(s)\le t\}.$$
\end{lemma}
\begin{proof}
  It follows from \eqref{exp} that $\tau(e^{t(x-E_{\rm Fix}x)-\varphi(t)})\le 1$. By \cite{JZ}*{Lemma 3.17}, we find
  $$\tau((x-E_{\rm Fix}x)\rho)\le \frac{\varphi(t)+\Ent(\rho)}t.$$
  Note that $\tau((E_{\rm Fix}x)\rho)=\tau(xE_{\rm Fix}\rho)$. Since $\varphi^*$ is continuous, taking inf over $t$ on the right-hand side gives the assertion.
\end{proof}

From now on, let us fix the Young function $\phi(t)=e^{t^2}-1$. This function is usually called $\psi_2$ in literature. It is well known that $\|f\|_\phi\le C_1$ is equivalent to (see, e.g., \cite{Ver})
 \begin{equation}\label{pbd}
 \|f\|_p\le C_2\sqrt{p} \mbox{ for $p\ge 1$}.
 \end{equation}
Given two positive $\tau$-measurable operators $\rho$ and $\si$ with $\tau(\rho)=\tau(\si)=1$, recall from Section 1 that
 \[
 Q_{\phi}(\rho,\si)=\sup \{|\tau(x\rho-x\si)|: x \mbox{ self-adjoint}, \|\Ga_A(x,x)^{1/2}\|_{\phi}\le 1\}.
 \]
 The definition of $Q_\phi$ is motivated by the functional representations of Wasserstein distance \eqref{w1de} and \eqref{w2fu}. Notice that $Q_{\phi}$ is much bigger than $Q_1$ in general. One may compare the following result with various transportation cost inequalities obtained in \cite{BGL}*{Section 5}.
\begin{theorem}\label{tran}
  Suppose the exponential inequality \eqref{exp2} holds for all self-adjoint $x$ with $E_{\Fix} x=0$ and $\|\Ga_A(x,x)^{1/2}\|_\phi\le 1$ in a noncommutative probability space $(\nx,\tau)$. Then there exist absolute positive constants $H_0,C_1,C_2,C_3$ such that for any positive $\tau$-measurable operator $\rho$ affiliated to $\nx$ with $\tau(\rho)=1$, we have
  \begin{equation}\label{teqn}
  Q_{\phi}(\rho, E_{\rm Fix}\rho)\le \begin{cases}
     C_1\sqrt{\Ent(\rho)}, \mbox{ if } \Ent(\rho) <H_0,\\
     C_2\Ent(\rho)+C_3, \mbox{ if } \Ent(\rho) \ge H_0.
    \end{cases}
  \end{equation}
  In particular, there exists $C'>0$ such that
  \[
  Q_{\phi}(\rho, E_{\rm Fix}\rho)\le C'\max\{\sqrt{\Ent(\rho)}, \Ent(\rho)\}.
  \]
\end{theorem}
\begin{proof}
  Assume $x$ is self-adjoint and $E_{\Fix}x=0$. We can rewrite \eqref{exp2} as $\tau(e^{tx})\le \tau(e^{ct^2\Ga_A(x,x)})$ for all $t\in \rz$. Using \eqref{pbd} and under the assumption $\|\Ga_A(x,x)^{1/2}\|_\phi\le 1$, we have
 \[
  \tau(e^{ct^2\Ga_A(x,x)})= \sum_{k=0}^\8 \frac{(ct^2)^k \tau(\Ga_A(x,x)^k)}{k!}\le 1+\sum_{k=1}^\8 \frac{(ct^2)^k (C_2\sqrt{2k})^{2k}}{k!}=:1+g(t).
 \]
 Let $C=2cC_2^2$. Then the series on the right-hand converges if $t<\frac1{\sqrt{Ce}}$. To get an explicit bound, note that for fixed $0<\eps<1$, there exists $C_\eps>0$ such that the series $g(t)$ converges uniformly and is bounded by $C_\eps t^2$ for $t\in[0,\frac{1-\eps}{\sqrt{Ce}}]$. We define
  \[
  f(t)= \begin{cases}
    C_\eps t^2, \mbox{ if } 0\le t\le \frac{1-\eps}{\sqrt{Ce}},\\
    \8, \mbox{ if } t> \frac{1-\eps}{\sqrt{Ce}}.
  \end{cases}
  \]
  Clearly $f(t)$ is a Young function and $g(t)\le f(t)$. The Legendre transform of $f$ is given by
  \[
  f^*(s) = \begin{cases}
    \frac{s^2}{4C_\eps},\quad 0\le s\le \frac{2C_\eps(1-\eps)}{\sqrt{Ce}},\\
    \frac{1-\eps}{\sqrt{Ce}}s-\frac{C_\eps(1-\eps)^2}{Ce}, \quad s> \frac{2C_\eps(1-\eps)}{\sqrt{Ce}}.
  \end{cases}
  \]
  We find the inverse function
  \[
  (f^*)^{-1}(z) = \begin{cases}
    2\sqrt{C_\eps z},\quad 0\le z < \frac{C_\eps (1-\eps)^2}{Ce},\\
    \frac{\sqrt{Ce}}{1-\eps}z+\frac{C_\eps (1-\eps)}{\sqrt{Ce}},\quad z\ge \frac{C_\eps (1-\eps)^2}{\sqrt{Ce}}.
    \end{cases}
  \]
  Since $\Ga_A(x,x)=\Ga_A(x-E_{\Fix}x, x-E_{\Fix}x)$, it suffices to take sup over all self-adjoint $x$ with $E_{\Fix}x=0$ in the definition of $Q_{\phi}$. The proof is complete by Lemma \ref{ent1}.
\end{proof}

\begin{exam}
  Consider the Gaussian space $(\rz,\ga)$ where $\ga$ is the standard Gaussian measure. Let $\mu$ be a probability measure absolutely continuous with respect to $\ga$.  In order to compute $Q_1(\ga,\mu)$, one takes supremum over essentially linear functions. On the other hand, one needs to take supremum over quadratic functions to compute $Q_{\phi}(\ga,\mu)$.

  Given $a\in\rz$, let us consider $\mu(B)=\ga(a+B)$ for any Borel set $B\subset \rz$ as suggested in \cite{Ta}. Let $f(x)=d\mu/d\ga=e^{ax-a^2/2}$. Then $\Ent(f)=\int f\ln f d\ga=a^2/2$. Let $g(x)=\frac{x^2-1}2$. Then $\int gd\ga=0$, $\|g'\|_\phi=2\sqrt{2}/\sqrt{3}$ and
  \[
  Q_\phi(\ga, \mu)=\sup_{\|g'\|_\phi\le 1} \int gf d\ga-\int g d\ga\ge \frac{\sqrt{3}a^2}{4\sqrt{2}}.
  \]
  This example shows that the estimate in \eqref{teqn} is sharp up to a constant, at least for large entropy.
\end{exam}

In fact, we have the following immediate application in the Gaussian setting. Let $H$ be a real separable Hilbert space. There exists a centered Gaussian family $W=\{W(h): h\in H\}$ defined on a probability space $(\Om,\fx,\pz)$ with $\ez(W(h)W(k))=\lge h, k\rge$. Let $\hx_n$ denote the Wiener chaos of order $n$, spanned by $\{H_n(W(h)): h\in H,\|h\|=1\}$ in $L_2(\Om,\fx,\pz)$, where $H_n(x)=\frac{(-1)^n e^{x^2/2}}{\sqrt{n!}}\frac{d^n}{dx^n}e^{-x^2/2}$ is the Hermite polynomial of order $n$. In particular, $H_1(x)=x$, $H_2(x)=\frac1{\sqrt{2}}(x^2-1)$. Let $\gx$ be the $\si$-algebra generated by $\{W(h):h\in H\}$. It is well known that $L_2(\Om,\gx,\pz)=\oplus_{n=0}^\8 \hx_n$. See \cite{Nua} for more details.
\begin{cor}\label{cha1}
Let $P_{\hx_n}: L_2(\Om,\gx,\pz)\to \hx_n$ be the orthogonal projection. Then for any positive function $f\in L_2(\Om,\gx,\pz)$ with $\|f\|_1=1$, we have
\[
\|P_{\hx_1}f\|_2\le C\sqrt{\Ent(f)},  \quad \|P_{\hx_2}(f)\|_2\le C'\max\{\sqrt{\Ent(f)},~\Ent(f)\}.
\]
\end{cor}
\begin{proof}
We first consider the second inequality. Let $(e_i)$ be an orthonormal basis of $H$. Then $$\{H_2(W(e_i))\}_i\cup\{H_1(W(e_i))H_1(W(e_j))\}_{i\neq j}$$
gives an orthonormal basis of $\hx_2$; see \cite{Nua}*{Proposition 1.1.1}. Write $g_i=W(e_i)$.  It suffices to show that
\[
\lge f, h\rge \le C' \max\{\sqrt{\Ent(f)},~\Ent(f)\}
\]
for all $h=\sum_{i=1}^\8 a_{ii} H_2(g_i) +\sum_{1\le i<j} a_{ij} H_1(g_i) H_1(g_j)$ where $\sum_{j\ge i\ge 1} a_{ij}^2=1$. Note that $(g_i)_i$ are independent standard Gaussian random variables. We have $\int h d\pz=0$. Consider $h$ as a function in the Gaussian space $(\rz^\nz,\ga)$. A computation yields
\[
\||\nabla h|\|_{L_p(\ga)}=\|(\sqrt{2}a_{ii}g_i+\sum_{j> i}a_{ij}g_j)_{i=1}^\8\|_{L_p(\ell_2)}.
\]
Then by the Minkowski inequality and the Khintchine inequality (or directly, the Khintchine--Kahane inequality), we have
\begin{align*}
  \| |\nabla h|\|_{L_p(\ga)}&\le \|(\sum_{j=i+1}^\8 a_{ij}g_j+\sqrt{2}a_{ii} g_i)_{i=1}^\8\|_{\ell_2(L_p)}\le  c_1 \sqrt{p}\Big(\sum_{i=1}^\8\sum_{j=i}^\8 |a_{ij}|^2\Big)^{1/2} = c_1\sqrt{p}.
\end{align*}
Hence, $\|\nabla h\|_\phi \le c$ for some numerical constant $c$. It follows that
\begin{align*}
  \|P_{\hx_2}(f)\|_2& \le \sup_{\|\nabla h \|_\phi \le c, \int h d\ga=0 }\lge f, h\rge_{L_2(\Om,\gx,\pz)} \le \sup_{\|\nabla h\|_\phi \le c} \int hf d\ga-\int h d\ga.
\end{align*}
Since the fixed point algebra for the Ornstein--Uhlenbeck semigroup is trivial, Proposition \ref{clpo} and Theorem \ref{tran} yield the second inequality. The first one follows from the same argument with the help of \eqref{wse2}.
\end{proof}
It is well known that the Walsh system (or Rademacher sequence) has similar properties to the Gaussian system. Let $(\eps_i)_{i\ge 1}$ be a Rademacher sequence. It can be realized as the coordinate functions of the discrete cube $\Om=\{-1,1\}^\nz$ with product uniform probability $\pz$. The Walsh system is given by $\{\eps_B=\prod_{j\in B} \eps_j | B\subset \nz, |B|<\8 \}$. It was shown in \cite{ELP} that the $L_p$ Poincar\'e inequalities hold for the Walsh system with the number operator. The same proof as for Corollary \ref{cha1} gives an estimate of Rademacher chaos of order 1 and 2. Let $\kx_2={\rm span}\{\eps_i\eps_j:i\neq j\}$ and $\kx_1={\rm span}\{\eps_i:i\in \nz\}$.
\begin{cor}
Let $P_{\kx_2}: L_2(\Om, \pz)\to \kx_2$ be the orthogonal projection. Then for any positive function $f\in L_2(\Om,\pz)$ with $\|f\|_1=1$,
\[
\|P_{\kx_1}(f)\|_2\le C\sqrt{\Ent(f)}.\quad  \|P_{\kx_2}(f)\|_2\le C'\max\{\sqrt{\Ent(f)},~\Ent(f)\}.
\]
\end{cor}

\begin{rem} Talagrand showed in \cite{Ta} that for Gaussian measure $\ga$ on $\rz^{\nz}$,
\begin{equation}\label{w2in}
  W_2(\ga,\mu)\le \sqrt{2cD(\mu||\ga)}
\end{equation}
for any probability measure $\mu$ absolutely continuous with respect to $\ga$ with $c=1$. An alternative proof was given in \cite{BG}. Otto--Villani \cite{OV} proved \eqref{w2in} in general Riemannian setting under the assumption of logarithmic Sobolev inequality. A simplified proof was given by Bobkov--Gentil--Ledoux in \cite{BGL}, still using LSI. It would be interesting to compare $Q_1$ and $Q_\phi$ with $W_1$ and $W_2$ in this commutative setting. Obviously, $Q_\phi\ge Q_1$, and in the Gaussian setting $Q_\phi$ is no less than $W_2$, because Talagrand's transportation cost inequality is sharp in this case and $Q_\phi$ is of the same order as $\Ent$ if the entropy is large. But the relationship between $Q_\phi$ and $W_2$ is not clear to us in general.

The interesting fact in \eqref{teqn} is that a phase transition may happen, which is different from Bobkov--G\"otze and Talagrand's inequality. As observed in \cite{JZ}*{Example 3.11}, LSI may not hold for the non-diffusion semigroups. Especially for the more general noncommutative setting, an estimate like \eqref{teqn} seems desirable without knowing further information such as LSI. In summary, the family of $(p,p)$ Poincar\'e inequalities indeed gives more information in transportation type inequalities compared with the $(p,\8)$ Poincar\'e inequalities obtained in \cite{JZ}. It is also reasonably effective in deducing transportation type inequalities compared with LSI.
\end{rem}

\section{New quantum metric spaces}
Compact quantum metric spaces are first introduced by Rieffel \cites{Ri1,Ri2}. Let $\ax$ be a $C^*$-algebra and $\bx$ a unital dense subalgebra of $\ax$. Let  $\opnorm{\cdot}$ be a seminorm satisfying the Leibniz condition
\begin{equation}\label{leib}
  \opnorm{ab}\le \opnorm{a}~\|b\|_\ax+\|a\|_\ax \opnorm{b}
\end{equation}
and
\begin{equation}\label{vani}
  \opnorm{a}=0 \mbox{ if and only if } a=c 1_\ax \mbox{ for some scalar } c.
\end{equation}
Then the triple $(\ax,\bx,\opnorm{\cdot})$ is called a \emph{compact quantum metric space} if the distance $d(\phi,\psi)=\sup\{|\phi(a)-\psi(a)|:\opnorm{a}\le 1\}$ induces the weak* topology on the state space $S(\ax)$. Ozawa and Rieffel made the following crucial observation in \cite{OR}.
\begin{lemma}\label{cqms}
  Suppose
  \[
  \{a\in \bx: \opnorm{a}\le 1, \si(a)=0\}
  \]
  is relatively compact in $\ax$ for some state $\si$. Then $(\ax,\bx,\opnorm{\cdot})$ is a compact quantum metric space.
\end{lemma}
It turns out that the norm $\|\Ga_A(f,f)^{1/2}\|_\phi$ we used in the transportation type inequality gives rise to new quantum metric spaces. Here $A$ is the generator of a general standard semigroup acting on a noncommutative probability space. Define
$$\opnorm{a}=\max\{\|\Ga_A(a,a)^{1/2}\|_\phi, \|\Ga_A(a^*,a^*)^{1/2}\|_\phi\}$$
for $a\in \bx$. We remark that $\opnorm{\cdot}$ does not satisfy the condition \eqref{vani} in general. Let $C_r^*(G)$ denote the reduced group $C^*$-algebra of $G$ and $\cz(G)$ the subset of finitely supported elements. In our examples later, $\ax=C^*(G)$ and $\bx=\cz(G)$.
\begin{lemma}\label{lip}
  $\opnorm{\cdot}$ satisfies the Leibniz condition.
\end{lemma}
\begin{proof}
The norm $\|\Ga_A(a,a)^{1/2}\|_\8$ was considered in \cite{JMe} to give examples of compact quantum metric spaces. The proof here is similar to \cite{JMe}*{Lemma 1.2.1} but more technical. We will need the theory of Hilbert $C^*$-modules. Our general reference is \cite{Lan95}. Roughly speaking, the idea is to connect $\Ga_A$-form with a derivation taking values in a certain Hilbert $C^*$-module to get the ``product rule'', and relate the module to a column space of a von Neumann algebra to get the triangle inequality.

Let us first recall that the norm of an element $x$ in a Hilbert $C^*$-module $H$ is given by $\|x\|=\||x|\|$ and $|x|=\lge x,x\rge^{1/2}$ belongs to the underlying $C^*$-algebra. In particular, if the $C^*$-algebra is the group von Neumann algebra $\lx(G)$ with the canonical trace $\tau$, we can construct the Orlicz space $L_\phi(\lx(G))$ and introduce the notation $\|x\|_{L_\phi(H, \tau)}:=\||x|\|_{L_\phi(\lx(G),\tau)}$ for $x\in H$.

Consider the KSGNS construction of Hilbert $\lx(G)$-module. For $\sum_{i=1}^n a_i\otimes x_i, \sum_{j=1}^m b_j\otimes y_j$ in the algebraic tensor product $\cz(G)\otimes \lx(G)$, define
\begin{equation}\label{inpr}
  \lge \sum_{i=1}^n a_i\otimes x_i, \sum_{j=1}^m b_j\otimes y_j\rge_{\lx(G)}=\sum_{i=1}^n\sum_{j=1}^m x_i\Ga_A(a_i,b_j)y_j.
\end{equation}
Denote by $H_\Ga$ the (norm) completion of the quotient of $\cz(G)\otimes \lx(G)$ by the kernel of $\lge\cdot, \cdot\rge_{\lx(G)}$. Then $H_\Ga$ is  a Hilbert $\lx(G)$-module with a natural right action from $\lx(G)$. For $a\in \cz(G)$, we define the left action
\begin{equation}\label{lact}
  \pi(a)\Big(\sum_j^m b_j\dot\otimes y_j\Big)= \sum_{j=1}^m(ab_j)\dot\otimes y_j.
\end{equation}
Here $b\dot\otimes y$ denotes the equivalent class of $b\otimes y$ in $H_\Ga$. With the left action \eqref{lact}, $H_\Ga$ becomes a normed left $\cz(G)$-module with
\begin{equation*}
  \Big\|\pi(a)\Big(\sum_j^m b_j\dot\otimes y_j\Big)\Big\|\le \|a\|\Big\|\sum_j^m b_j\dot\otimes y_j\Big\|.
\end{equation*}
It follows that $\pi(a)$ extends to an element of $L(H_\Ga)$, where $L(H_\Ga)$ denotes the adjointable maps of $H_\Ga$. It is also clear from the definition that $\pi$ extends to a unital $*$-homomorphism $\pi: C_r^*(G)\to L(H_\Ga)$. In this way, $H_\Ga$ becomes a  $C^*_r\dash\lx(G)$ bimodule. In what follows, we do not write $\pi$ explicitly. Since $\cz(G)$ is countable, $H_\Ga$ is countably generated. We deduce from Kasparov's absorption theorem (\cite{Lan95}*{Theorem 6.2}) that there exists a unitary right module map $u: H_\Ga\oplus C(\lx(G))\to C(\lx(G))$. Here $C(\lx(G))=\ell_2\otimes \lx(G)$ is the Hilbert $\lx(G)$-module formed by the separable Hilbert space $\ell_2$ and $\lx(G)$. As observed in \cite{PX}, $C(\lx(G))$ is a column subspace of $B(\ell_2)\overline\otimes \lx(G)$, on which there is a natural semifinite trace $tr\otimes\tau$, where $tr$ is the usual trace on $B(\ell_2)$. Consider the Orlicz space $L_\phi(B(\ell_2)\overline\otimes \lx(G))$. Then the norm $\|\cdot\|_{L_\phi(tr\otimes\tau)}$ restricted to $C(\lx(G))$ gives a norm
\begin{equation}\label{bana}
  \|(a_k)\|_{L_\phi(tr\otimes \tau)}= \|(\sum_{k}|a_k|^2)^{1/2}\|_{L_\phi(\tau)} = \|\lge (a_k),(a_k)\rge^{1/2} \|_{L_\phi(\tau)}
\end{equation}
for $(a_k)\in C(\lx(G))$, where the inner product is from the Hilbert $\lx(G)$-module structure of $C(\lx(G))$. Define $v: H_\Ga\to C(\lx(G))$ to be the restriction of $u$ on the first coordinate of $u$. Then $v$ is an injective right module map which preserves the inner product of $H_\Ga$. Since by definition $\|\xi\|_{L_\phi(H_\Ga,\tau)}=\||\xi|\|_{L_\phi(\lx(G),\tau)}$,  \eqref{bana} implies that
\begin{equation}\label{isom}
\|\xi\|_{L_\phi(H_\Ga,\tau)}=\|\lge\xi,\xi\rge^{1/2}\|_{L_\phi(\lx(G),\tau)} =\|\lge v(\xi),v(\xi)\rge^{1/2}\|_{L_\phi(\lx(G),\tau)} =\|v(\xi)\|_{L_\phi(tr\otimes \tau)}
\end{equation}
for $\xi\in H_\Ga$.

Now we define $\de(a)=a\otimes 1-1\otimes a$ for $a\in\cz(G)$. $\de$ takes values in a Hilbert submodule $H_0$ of $H_\Ga$, which can be obtained by running the same KSGNS construction for $\{\sum_i a_i\otimes x_i:\sum_i a_ix_i=0, a_i\in \cz(G),x_i\in \lx(G)\}$ as in \eqref{inpr} and \eqref{lact}. In fact, $H_0$ was considered in \cite{Pet} with completion in a different topology, but $\de$ is still a derivation and by \eqref{inpr} $\lge \de(a),\de(a)\rge =\Ga_A(a,a)$, as checked in \cite{JMe}. For $a,b\in\cz(G)$ and $\xi\in H_\Ga$, we have by \eqref{lact} and \cite{Lan95}*{Proposition 1.2} that $\lge a\xi,a\xi\rge=\lge \pi(a)\xi, \pi(a)\xi\rge \le \|a\|_\8^2 |\xi|^2$ and thus
\begin{equation}\label{ine1}
  \|a\xi\|_{L_\phi(H_\Ga,\tau)}\le \|a\|_\8 \|\xi\|_{L_\phi(H_\Ga,\tau)}.
\end{equation}
Write the diagonal matrix with $b$ on the diagonal as $id \otimes b\in B(\ell_2)\overline\otimes \lx(G)$. Since $v$ is a right module map, using \eqref{isom} and Proposition \ref{orli}, we have
\begin{equation}\label{ine2}
  \begin{split}
  \|\xi b\|_{L_\phi(H_\Ga, \tau)}&=\|v(\xi)b\|_{L_\phi(tr\otimes \tau)}=\|v(\xi)(id \otimes b)\|_{L_\phi(tr\otimes \tau)} \\ &\le\|v(\xi)\|_{L_\phi(tr\otimes \tau)}\|id\otimes b\|_\8=\|\xi \|_{L_\phi(H_\Ga, \tau)}\|b\|_\8.
\end{split}
\end{equation}
By \eqref{isom}, \eqref{ine1}, \eqref{ine2} and the triangle inequality of $C(\lx(G))$, we have
\begin{align*}
  \|\Ga_A(ab,ab)^{1/2}\|_{L_\phi(\lx(G),\tau)}&=\|\lge \de(ab),\de(ab)\rge^{1/2}\|_{L_\phi(\lx(G),\tau)} =\|\de(ab)\|_{L_\phi(H_\Ga,\tau)}\\ &=\|v(\de(ab))\|_{L_\phi(tr\otimes\tau)} =\|v(a\de(b))+v(\de(a)b)\|_{L_\phi(tr\otimes\tau)}\\
  &\le \|v(a\de(b))\|_{L_\phi(tr\otimes\tau)} +\|v(\de(a)b)\|_{L_\phi(tr\otimes\tau)}\\
  &= \| a\de(b)\|_{L_\phi(H_\Ga,\tau)}+\|\de(a)b\|_{L_\phi(H_\Ga, \tau)}\\
  &\le \|a\|_\8\|\de(b)\|_{L_\phi(H_\Ga,\tau)}+\|\de(a)\|_{L_\phi(H_\Ga,\tau)} \|b\|_\8.
\end{align*}
We have shown that $\|\Ga_A(ab,ab)^{1/2}\|_\phi\le \|a\|_\8\|\Ga_A(b,b)^{1/2}\|_\phi + \|\Ga_A(a,a)^{1/2}\|_\phi \|b\|_\8$. Similarly, we have $\|\Ga_A(b^*a^*,b^*a^*)^{1/2}\|_\phi\le \|a^*\|_\8\|\Ga_A(b^*,b^*)^{1/2}\|_\phi + \|\Ga_A(a^*,a^*)^{1/2}\|_\phi \|b^*\|_\8$. Therefore, $\opnorm{\cdot}$ satisfies the Leibniz condition.
\end{proof}

Let us recall some notation. $b_\psi: G\to H_\psi$ is the $1$-cocycle associated to the cn-length function $\psi$. $T_t=e^{-t A}$ where $T_t \la(g)=e^{-t\psi(g)}\la(g)$. A finitely generated group $G$ is said to have rapid $s$-decay if $\|x\|_\8\le C_s (1+k)^s\|x\|_2$ for every $x=\sum_{|g|=k}a_g\la(g)$ and some $s<\8$ where $|g|$ is the length of the reduced word $g$ with respect to fixed generators. $L_p^0(\nx)$ denotes the complemented subspace of $L_p(\nx)$ of elements for which $\lim_{t\to \8} T_t x = 0$. For instance, $x-E_{\rm Fix} x\in L_p^0$ for $x\in L_p$.

Using the theory of noncommutative Riesz transform and completely bounded multipliers, it was shown in \cite{JMe}*{Theorem 1.2.3} and \cite{JMP}*{Corollary 5.12} that $(C_r^*(G),\cz(G), \|\cdot\|_\Ga)$ is a compact quantum metric space under certain conditions where
$$\|f\|_\Ga= \max\{\|\Ga_\psi(f,f)^{1/2}\|_\8, \|\Ga_\psi(f^*,f^*)^{1/2}\|_\8\}.$$
In fact, once we notice the fact that $\|\Ga_\psi(f,f)^{1/2}\|_p\le C_p\| \Ga_\psi(f,f)^{1/2} \|_\phi \le C_p'\|\Ga_\psi(f,f)^{1/2}\|_\8$, the same proofs also work if we replace $\|\cdot\|_\Ga$ by $\opnorm{\cdot}$ defined above. We add a sketch of proof for completeness. The interested readers are referred to \cite{JMe}*{Corollary 2} and \cite{JMP}*{Corollary 5.12} for details. We remark that by the proof of \cite{JMP}*{Corollary 5.12}, the condition \eqref{vani} in this context is equivalent to the condition $\ker{\psi}=\{e\}$, where $e$ is the identity element of $G$.
\begin{cor}\label{qms}
  $(C_r^*(G),\cz(G),\opnorm{\cdot})$ is a compact quantum metric space provided one of the following holds:
  \begin{enumerate}
    \item $G$ is finitely generated with rapid $s$-decay and $\inf_{|g|=k}\psi(g)\ge c_\al(1+k)^\al$ for some $\al>0$.
    \item ${\rm dim} H_\psi <\8$, the kernel of $\psi$ is $\{e\}$ and $\inf_{b_\psi(g)\neq 0}\psi(g)>0$.
  \end{enumerate}
\end{cor}
\begin{proof}[Sketch of proof]
The assumption in (1) implies that $\ker(\psi)=\{e\}$. By \cite{JMe}*{Lemma 1.3.1}, it also implies
  \[
  \|T_t: L_2^0(\lx(G))\to \lx(G)\|\le C(s,\al)t^{-n(s,\al)/4}
   \]
   for some $n>0$ and $A^{-1}$ is compact on $L_2^0(\lx(G))$. Then by \cite{JMe}*{Theorem 1.1.7}, $A^{-z}: L_p^0\to L_\8^0$ is compact for $z>n/(4p)$. In particular,
   \[
   \{x\in L_\8^0(\lx(G)):\|A^{z}x\|_p\le 1 \}\subset L_\8^0(\lx(G))
   \]
   is relatively compact in $\lx(G)$ for $z>n/(4p)$. Then by the boundedness of noncommutative Riesz transform \cite{JMe}*{Theorem 2.5.13},
   \begin{align*}
        \|A^{1/2}x\|_p &\le C(p)\max\{\|\Ga_\psi(x,x)^{1/2}\|_p,\|\Ga_\psi(x^*,x^*)^{1/2}\|_p\}\\
        &\le C'(p)\max\{\|\Ga_\psi(x,x)^{1/2}\|_\phi, \|\Ga_\psi(x^*,x^*)^{1/2}\|_\phi\}.
   \end{align*}
    for $p\ge 2$. Choosing $p>n/2$ (e.g., $p=n$ and $z=1/2$), Lemma \ref{cqms} and Lemma \ref{lip} yield (1).

  The argument for (2) is the same as that of \cite{JMP}*{Corollary 5.12}. The assumption implies $A^{-1}:L_2^0(\lx(G))\to L_2^0(\lx(G))$ is a compact operator. Combining the theory of completely bounded multipliers and some abstract semigroup theory, it can be shown that $A^{-1/2}: L_p^0(\lx(G))\to L_\8^0(\lx(G))$ is also compact for large $p$. Then by the same argument as for (1), $\{x\in L_\8(\lx(G)): \opnorm{x}\le 1, \tau(x)=0\}$ is relatively compact in $\lx(G)$ and the proof is complete, by Lemma \ref{cqms}.
\end{proof}
\begin{rem}
\begin{enumerate}
  \item The proof of Corollary \ref{qms} shows that any seminorm of the form $\|f\|_{\Ga,\al}=\max\{\|\Ga_A(f,f)^{1/2}\|_\al, \|\Ga_A(f^*,f^*)^{1/2}\|_\al\}$, which satisfies \eqref{leib}, \eqref{vani} and $\|x\|_p\le C(\al)\|x\|_\al$, will provide new examples of compact quantum metric spaces.
  \item From noncommutative geometric point of view, \eqref{wse2} and \eqref{teqn} may be regarded as transportation cost inequalities for quantum metric spaces.
\end{enumerate}
\end{rem}

\section*{Acknowledgements}
The author would like to thank Marius Junge for many inspiring conversations and suggestions, David Nualart for his comments on the results of this paper. This paper was completed while the author was visiting ICMAT in Spain. He sincerely appreciates the hospitality of the institute. Special thanks go to the anonymous referee for very careful reading and very helpful suggestions on improving both the mathematics and the presentation of the paper, whose opinion has been incorporated in the current version.
\bibliographystyle{plain}
\bibliography{pcr}

\begin{bibdiv}
\begin{biblist}

\bib{BC}{article}{
      author={Bekjan, Turdebek~N.},
      author={Chen, Zeqian},
       title={Interpolation and {$\Phi$}-moment inequalities of noncommutative
  martingales},
        date={2012},
        ISSN={0178-8051},
     journal={Probab. Theory Related Fields},
      volume={152},
      number={1-2},
       pages={179\ndash 206},
         url={http://dx.doi.org/10.1007/s00440-010-0319-2},
      review={\MR{2875756}},
}

\bib{BG}{article}{
      author={Bobkov, S.~G.},
      author={G{\"o}tze, F.},
       title={Exponential integrability and transportation cost related to
  logarithmic {S}obolev inequalities},
        date={1999},
        ISSN={0022-1236},
     journal={J. Funct. Anal.},
      volume={163},
      number={1},
       pages={1\ndash 28},
         url={http://dx.doi.org/10.1006/jfan.1998.3326},
      review={\MR{1682772 (2000b:46059)}},
}

\bib{BGL}{article}{
      author={Bobkov, Sergey~G.},
      author={Gentil, Ivan},
      author={Ledoux, Michel},
       title={Hypercontractivity of {H}amilton-{J}acobi equations},
        date={2001},
        ISSN={0021-7824},
     journal={J. Math. Pures Appl. (9)},
      volume={80},
      number={7},
       pages={669\ndash 696},
         url={http://dx.doi.org/10.1016/S0021-7824(01)01208-9},
      review={\MR{1846020 (2003b:47073)}},
}

\bib{BKS}{article}{
      author={Bo{\.z}ejko, Marek},
      author={K{\"u}mmerer, Burkhard},
      author={Speicher, Roland},
       title={{$q$}-{G}aussian processes: non-commutative and classical
  aspects},
        date={1997},
        ISSN={0010-3616},
     journal={Comm. Math. Phys.},
      volume={185},
      number={1},
       pages={129\ndash 154},
         url={http://dx.doi.org/10.1007/s002200050084},
      review={\MR{1463036 (98h:81053)}},
}

\bib{BO}{book}{
      author={Brown, Nathanial~P.},
      author={Ozawa, Narutaka},
       title={{$C^*$}-algebras and finite-dimensional approximations},
      series={Graduate Studies in Mathematics},
   publisher={American Mathematical Society},
     address={Providence, RI},
        date={2008},
      volume={88},
        ISBN={978-0-8218-4381-9; 0-8218-4381-8},
      review={\MR{2391387 (2009h:46101)}},
}

\bib{Val}{book}{
      author={Cherix, Pierre-Alain},
      author={Cowling, Michael},
      author={Jolissaint, Paul},
      author={Julg, Pierre},
      author={Valette, Alain},
       title={Groups with the {H}aagerup property},
      series={Progress in Mathematics},
   publisher={Birkh\"auser Verlag},
     address={Basel},
        date={2001},
      volume={197},
        ISBN={3-7643-6598-6},
         url={http://dx.doi.org/10.1007/978-3-0348-8237-8},
        note={Gromov's a-T-menability},
      review={\MR{1852148 (2002h:22007)}},
}

\bib{PS}{article}{
      author={Dirksen, Sjoerd},
      author={de~Pagter, Ben},
      author={Potapov, Denis},
      author={Sukochev, Fedor},
       title={Rosenthal inequalities in noncommutative symmetric spaces},
        date={2011},
        ISSN={0022-1236},
     journal={J. Funct. Anal.},
      volume={261},
      number={10},
       pages={2890\ndash 2925},
         url={http://dx.doi.org/10.1016/j.jfa.2011.07.015},
      review={\MR{2832586 (2012k:46073)}},
}

\bib{Eff}{article}{
      author={Effros, Edward~G.},
       title={A matrix convexity approach to some celebrated quantum
  inequalities},
        date={2009},
        ISSN={1091-6490},
     journal={Proc. Natl. Acad. Sci. USA},
      volume={106},
      number={4},
       pages={1006\ndash 1008},
         url={http://dx.doi.org/10.1073/pnas.0807965106},
      review={\MR{2475796 (2010a:81028)}},
}

\bib{ELP}{article}{
      author={Efraim, L.~Ben},
      author={Lust-Piquard, F.},
       title={Poincar\'e type inequalities on the discrete cube and in the
  {CAR} algebra},
        date={2008},
        ISSN={0178-8051},
     journal={Probab. Theory Related Fields},
      volume={141},
      number={3-4},
       pages={569\ndash 602},
         url={http://dx.doi.org/10.1007/s00440-007-0094-x},
      review={\MR{2391165 (2009c:46085)}},
}

\bib{FK}{article}{
      author={Fack, Thierry},
      author={Kosaki, Hideki},
       title={Generalized {$s$}-numbers of {$\tau$}-measurable operators},
        date={1986},
        ISSN={0030-8730},
     journal={Pacific J. Math.},
      volume={123},
      number={2},
       pages={269\ndash 300},
         url={http://projecteuclid.org/getRecord?id=euclid.pjm/1102701004},
      review={\MR{840845 (87h:46122)}},
}

\bib{Fan}{book}{
      author={Fang, Shizan},
       title={Introduction to {M}alliavin calculus},
      series={Mathematics Series for Graduate Student},
   publisher={Tsinghua University Press and Springer},
     address={Beijing},
        date={2005},
      volume={3},
        ISBN={7-302-09181-1},
}

\bib{Gro}{article}{
      author={Gross, Leonard},
       title={Logarithmic {S}obolev inequalities},
        date={1975},
        ISSN={0002-9327},
     journal={Amer. J. Math.},
      volume={97},
      number={4},
       pages={1061\ndash 1083},
      review={\MR{0420249 (54 \#8263)}},
}

\bib{JMP}{article}{
      author={{Junge}, M.},
      author={{Mei}, T.},
      author={{Parcet}, J.},
       title={{Smooth Fourier multipliers on group von Neumann algebras}},
        date={2010-10},
     journal={ArXiv e-prints},
      eprint={1010.5320},
}

\bib{JZ}{article}{
      author={{Junge}, M.},
      author={{Zeng}, Q.},
       title={{Noncommutative martingale deviation and Poincar\'e type
  inequalities with applications}},
        date={2012-11},
     journal={ArXiv e-prints},
      eprint={1211.3209},
}

\bib{JMe}{article}{
      author={Junge, Marius},
      author={Mei, Tao},
       title={Noncommutative {R}iesz transforms---a probabilistic approach},
        date={2010},
        ISSN={0002-9327},
     journal={Amer. J. Math.},
      volume={132},
      number={3},
       pages={611\ndash 680},
         url={http://dx.doi.org/10.1353/ajm.0.0122},
      review={\MR{2666903 (2012b:46136)}},
}

\bib{JX07}{article}{
      author={Junge, Marius},
      author={Xu, Quanhua},
       title={Noncommutative maximal ergodic theorems},
        date={2007},
        ISSN={0894-0347},
     journal={J. Amer. Math. Soc.},
      volume={20},
      number={2},
       pages={385\ndash 439},
         url={http://dx.doi.org/10.1090/S0894-0347-06-00533-9},
      review={\MR{2276775 (2007k:46109)}},
}

\bib{JZ1}{article}{
      author={{Junge}, Marius},
      author={{Zeng}, Qiang},
       title={{Noncommutative Bennett and Rosenthal Inequalities}},
        date={to appear},
     journal={Ann. Probab.},
      eprint={arXiv 1111.1027},
}

\bib{KS}{article}{
      author={Kalton, N.~J.},
      author={Sukochev, F.~A.},
       title={Symmetric norms and spaces of operators},
        date={2008},
        ISSN={0075-4102},
     journal={J. Reine Angew. Math.},
      volume={621},
       pages={81\ndash 121},
         url={http://dx.doi.org/10.1515/CRELLE.2008.059},
      review={\MR{2431251 (2009i:46118)}},
}

\bib{Lan95}{book}{
      author={Lance, E.~C.},
       title={Hilbert {$C^*$}-modules},
      series={London Mathematical Society Lecture Note Series},
   publisher={Cambridge University Press},
     address={Cambridge},
        date={1995},
      volume={210},
        ISBN={0-521-47910-X},
         url={http://dx.doi.org/10.1017/CBO9780511526206},
        note={A toolkit for operator algebraists},
      review={\MR{1325694 (96k:46100)}},
}

\bib{Mar}{article}{
      author={Marton, K.},
       title={A simple proof of the blowing-up lemma},
        date={1986},
        ISSN={0018-9448},
     journal={IEEE Trans. Inform. Theory},
      volume={32},
      number={3},
       pages={445\ndash 446},
         url={http://dx.doi.org/10.1109/TIT.1986.1057176},
      review={\MR{838213 (87e:94018)}},
}

\bib{Mi}{article}{
      author={Milman, Emanuel},
       title={Properties of isoperimetric, functional and transport-entropy
  inequalities via concentration},
        date={2012},
        ISSN={0178-8051},
     journal={Probab. Theory Related Fields},
      volume={152},
      number={3-4},
       pages={475\ndash 507},
         url={http://dx.doi.org/10.1007/s00440-010-0328-1},
      review={\MR{2892954}},
}

\bib{Nua}{book}{
      author={Nualart, David},
       title={The {M}alliavin calculus and related topics},
     edition={Second},
      series={Probability and its Applications (New York)},
   publisher={Springer-Verlag},
     address={Berlin},
        date={2006},
        ISBN={978-3-540-28328-7; 3-540-28328-5},
      review={\MR{2200233 (2006j:60004)}},
}

\bib{OV}{article}{
      author={Otto, F.},
      author={Villani, C.},
       title={Generalization of an inequality by {T}alagrand and links with the
  logarithmic {S}obolev inequality},
        date={2000},
        ISSN={0022-1236},
     journal={J. Funct. Anal.},
      volume={173},
      number={2},
       pages={361\ndash 400},
         url={http://dx.doi.org/10.1006/jfan.1999.3557},
      review={\MR{1760620 (2001k:58076)}},
}

\bib{OR}{article}{
      author={Ozawa, Narutaka},
      author={Rieffel, Marc~A.},
       title={Hyperbolic group {$C^*$}-algebras and free-product
  {$C^*$}-algebras as compact quantum metric spaces},
        date={2005},
        ISSN={0008-414X},
     journal={Canad. J. Math.},
      volume={57},
      number={5},
       pages={1056\ndash 1079},
         url={http://dx.doi.org/10.4153/CJM-2005-040-0},
      review={\MR{2164594 (2006e:46080)}},
}

\bib{Pet}{article}{
      author={Peterson, Jesse},
       title={A 1-cohomology characterization of property ({T}) in von
  {N}eumann algebras},
        date={2009},
        ISSN={0030-8730},
     journal={Pacific J. Math.},
      volume={243},
      number={1},
       pages={181\ndash 199},
         url={http://dx.doi.org/10.2140/pjm.2009.243.181},
      review={\MR{2550142 (2010f:22009)}},
}

\bib{Pi}{incollection}{
      author={Pisier, Gilles},
       title={Riesz transforms: a simpler analytic proof of {P}.-{A}. {M}eyer's
  inequality},
        date={1988},
   booktitle={S\'eminaire de {P}robabilit\'es, {XXII}},
      series={Lecture Notes in Math.},
      volume={1321},
   publisher={Springer},
     address={Berlin},
       pages={485\ndash 501},
         url={http://dx.doi.org/10.1007/BFb0084154},
      review={\MR{960544 (89m:60178)}},
}

\bib{PX}{article}{
      author={Pisier, Gilles},
      author={Xu, Quanhua},
       title={Non-commutative martingale inequalities},
        date={1997},
        ISSN={0010-3616},
     journal={Comm. Math. Phys.},
      volume={189},
      number={3},
       pages={667\ndash 698},
         url={http://dx.doi.org/10.1007/s002200050224},
      review={\MR{1482934 (98m:46079)}},
}

\bib{Ra}{book}{
      author={Rachev, Svetlozar~T.},
       title={Probability metrics and the stability of stochastic models},
      series={Wiley Series in Probability and Mathematical Statistics: Applied
  Probability and Statistics},
   publisher={John Wiley \& Sons Ltd.},
     address={Chichester},
        date={1991},
        ISBN={0-471-92877-1},
      review={\MR{1105086 (93b:60012)}},
}

\bib{Ri1}{article}{
      author={Rieffel, Marc~A.},
       title={Metrics on states from actions of compact groups},
        date={1998},
        ISSN={1431-0635},
     journal={Doc. Math.},
      volume={3},
       pages={215\ndash 229 (electronic)},
      review={\MR{1647515 (99k:46126)}},
}

\bib{Ri2}{article}{
      author={Rieffel, Marc~A.},
       title={Group {$C^*$}-algebras as compact quantum metric spaces},
        date={2002},
        ISSN={1431-0635},
     journal={Doc. Math.},
      volume={7},
       pages={605\ndash 651 (electronic)},
      review={\MR{2015055 (2004k:22009)}},
}

\bib{Str}{book}{
      author={Stroock, Daniel~W.},
       title={Probability theory},
     edition={Second},
   publisher={Cambridge University Press},
     address={Cambridge},
        date={2011},
        ISBN={978-0-521-13250-3},
        note={An analytic view},
      review={\MR{2760872 (2012a:60003)}},
}

\bib{Tak}{book}{
      author={Takesaki, M.},
       title={Theory of operator algebras. {I}},
      series={Encyclopaedia of Mathematical Sciences},
   publisher={Springer-Verlag},
     address={Berlin},
        date={2002},
      volume={124},
        ISBN={3-540-42248-X},
        note={Reprint of the first (1979) edition, Operator Algebras and
  Non-commutative Geometry, 5},
      review={\MR{1873025 (2002m:46083)}},
}

\bib{Ta}{article}{
      author={Talagrand, M.},
       title={Transportation cost for {G}aussian and other product measures},
        date={1996},
        ISSN={1016-443X},
     journal={Geom. Funct. Anal.},
      volume={6},
      number={3},
       pages={587\ndash 600},
         url={http://dx.doi.org/10.1007/BF02249265},
      review={\MR{1392331 (97d:60029)}},
}

\bib{Tro}{article}{
      author={Tropp, Joel~A.},
       title={From joint convexity of quantum relative entropy to a concavity
  theorem of {L}ieb},
        date={2012},
        ISSN={0002-9939},
     journal={Proc. Amer. Math. Soc.},
      volume={140},
      number={5},
       pages={1757\ndash 1760},
         url={http://dx.doi.org/10.1090/S0002-9939-2011-11141-9},
      review={\MR{2869160 (2012j:52023)}},
}

\bib{Ver}{incollection}{
      author={Vershynin, Roman},
       title={Introduction to the non-asymptotic analysis of random matrices},
        date={2012},
   booktitle={Compressed sensing},
   publisher={Cambridge Univ. Press},
     address={Cambridge},
       pages={210\ndash 268},
      review={\MR{2963170}},
}

\bib{Vil}{book}{
      author={Villani, C{\'e}dric},
       title={Optimal transport},
      series={Grundlehren der Mathematischen Wissenschaften [Fundamental
  Principles of Mathematical Sciences]},
   publisher={Springer-Verlag},
     address={Berlin},
        date={2009},
      volume={338},
        ISBN={978-3-540-71049-3},
         url={http://dx.doi.org/10.1007/978-3-540-71050-9},
        note={Old and new},
      review={\MR{2459454 (2010f:49001)}},
}

\bib{Xu}{article}{
      author={Xu, Quan~Hua},
       title={Analytic functions with values in lattices and symmetric spaces
  of measurable operators},
        date={1991},
        ISSN={0305-0041},
     journal={Math. Proc. Cambridge Philos. Soc.},
      volume={109},
      number={3},
       pages={541\ndash 563},
         url={http://dx.doi.org/10.1017/S030500410006998X},
      review={\MR{1094753 (92g:46036)}},
}

\end{biblist}
\end{bibdiv}
\end{document}